\documentclass{modlms}

\usepackage{amssymb,amsmath,amscd,xspace,paralist}
\usepackage{graphicx}

\newtheorem{theorem}{Theorem}
\newtheorem{corollary}[theorem]{Corollary}
\newtheorem{lemma}[theorem]{Lemma}

\newnumbered{remark}[theorem]{Remark}
\newnumbered{remarks}[theorem]{Remarks}
\newnumbered{example}[theorem]{Example}
\newnumbered{examples}[theorem]{Examples}
\newnumbered{defn}[theorem]{Definition}
\newnumbered{defns}[theorem]{Definitions}

\newcommand{\Q}{{\mathbb{Q}}}
\newcommand{\N}{{\mathbb{N}}}
\newcommand{\R}{{\mathbb{R}}}
\newcommand{\RP}{{\R\mathrm{P}}}
\newcommand{\cA}{{\mathcal{A}}}
\newcommand{\IM}{{\mathcal{I}}}
\newcommand{\cO}{{\mathcal{O}}}
\newcommand{\MP}{{\mathcal{P}}}
\newcommand{\cR}{{\mathcal{R}}}
\newcommand{\cM}{{\mathcal{M}}}
\newcommand{\Wmap}{S}
\newcommand{\bal}{{\boldsymbol\alpha}}
\newcommand{\balF}{{\boldsymbol a}}
\newcommand{\bbeta}{{\boldsymbol\beta}}
\newcommand{\bgamma}{{\boldsymbol\gamma}}
\newcommand{\brho}{{\boldsymbol\rho}}
\newcommand{\bm}{{\mathbf m}}
\newcommand{\bn}{{\mathbf n}}
\newcommand{\topL}[1]{{\stackrel{\KF_{#1}}{\,\,\,\longrightarrow\,\,\,}}}
\newcommand{\Reg}{{\mathrm{Reg}}}

\newcommand{\oD}{{\overline{\Delta}}}

\newcommand{\wh}[1]{{\widehat{#1}}}
\newcommand{\RF}{\wh{\cR}}
\newcommand{\MF}{\wh{\cM}}
\newcommand{\IF}{\wh{\IM}}
\newcommand{\KF}{\wh{K}}
\newcommand{\DelF}{\wh{\Delta}}
\newcommand{\PhiF}{\wh{\Phi}}
\newcommand{\alphaF}{a}

\title{Symbol ratio minimax sequences in the lexicographic order}
\shorttitle{Symbol ratio minimax sequences}

\date{January 2013}

\author{Philip Boyland,
Andr\'e de Carvalho, and Toby Hall}

 \classno{37B10, %Symbolic Dynamics 
 37E45, % (Rotation numbers and vectors), 
 68R15% (Combinatorics on words)
}

\extraline{ The authors are grateful for the support of FAPESP grants
  2010/09667-0 and 2011/17581-0. This research has been supported
  in part by EU Marie-Curie IRSES Brazilian-European partnership in
  Dynamical Systems (FP7-PEOPLE-2012-IRSES 318999 BREUDS). AdC is partially supported by CNPq
  grant 308455/2010-0.\par\medskip\noindent {\tiny\phantom{x} \quad
    PB: Department of Mathematics, University of Florida, 372 Little
    Hall, Gainesville, FL 32611-8105, USA\\\phantom{x} \quad AdC:
    Departamento de Matem\'atica Aplicada, IME-USP, Rua Do Mat\~ao
    1010, Cidade Universit\'aria, 05508-090 S\~ao Paulo SP,
    Brazil\\ \phantom{x} \quad TH: Department of Mathematical
    Sciences, University of Liverpool, Liverpool L69 7ZL, UK}}

\begin{document}
\maketitle

\begin{abstract}
Consider the space of sequences of~$k$ letters ordered
lexicographically.  We study the set~$\cM(\bal)$ of all maximal
sequences for which the asymptotic proportions~$\bal$ of the letters
are prescribed, where a sequence is said to be maximal if it is at
least as great as all of its tails. The infimum of $\cM(\bal)$ is
called the {\em $\bal$-infimax} sequence, or the {\em $\bal$-minimax}
sequence if the infimum is a minimum. We give an algorithm which
yields all infimax sequences, and show that the infimax is {\em not} a
minimax if and only if it is the $\bal$-infimax for every $\bal$ in a
simplex of dimension~1 or greater. These results have applications
to the theory of rotation sets of beta-shifts and torus
homeomorphisms.
\end{abstract}

\section{Introduction}
Symbolic dynamics is a fundamental tool in dynamical systems theory,
and the interaction between the dynamics of the shift map and an order
structure is frequently important. For example, kneading
theory~\cite{kneading} describes the dynamics of a unimodal map as the
set of sequences which are less than or equal to the kneading sequence
of the map in the unimodal order; while in Parry's work~\cite{parry}
on beta-shifts it is the relationship between the shift map and the
lexicographic order which plays a central r\^ole. In such systems, a
particular orbit is present if the maximum (or more generally
supremum) of the orbit is less than or equal to a given sequence:
hence, in order to decide whether or not a given dynamical feature is
present, the key question is the size of the minimum, or infimum, of
the set of maximal sequences which exhibit the feature. It is for this
reason that such {\em minimax} and {\em infimax} sequences are
important.

This paper provides a description of minimax and infimax sequences in
the lexicographic order, where the relevant dynamical feature --
closely related to rotation vectors -- is the asymptotic proportions
of the letters. In the remainder of the introduction we will first
give an informal description of the main results, and then expand on
their dynamical significance.

Given~$k\ge 2$, let $\Sigma = \{1,2,\ldots,k\}^\N$ be the space of
sequences in the letters $1,2,\ldots,k$, ordered lexicographically,
and let $\sigma\colon\Sigma\to\Sigma$ be the shift map. A sequence
$w\in\Sigma$ is said to be {\em maximal} if $\sigma^r(w)\le w$ for
all~$r\ge 0$. 

We are interested in maximal sequences for which the asymptotic
proportions of the letters are given by some $\bal\in\Delta$,
where~$\Delta$ is the set of vectors in~$\R^k$ with non-negative
entries summing to~$1$. Denote by~$\cM(\bal)$ the subset of~$\Sigma$
consisting of maximal sequences~$w$ with the property that, for each
$i$ with $1\le i\le k$, the asymptotic proportion of the letter~$i$
in~$w$ is given by~$\alpha_i$. Let~$\IM(\bal)$ denote the infimum of
the set~$\cM(\bal)$, the {\em $\bal$-infimax} sequence. This infimum,
while necessarily maximal, need not in general be an element
of~$\cM(\bal)$: when it is, it is called the {\em $\bal$-minimax}
sequence.

\medskip

The main results of the paper can be summarised as follows:

\bigskip

\noindent\textsc{Theorem~\ref{thm:infimax}. (Description of Infimaxes)}
\,\, \emph{There is an algorithm for computing~$\IM(\bal)$ (to an
  arbitrary number of letters) in terms of a sequence of
  substitutions. This sequence of substitutions is determined by the
  itinerary of~$\bal$ under a multi-dimensional continued fraction map
  $K\colon \Delta\to\Delta$.}

\bigskip

\noindent\textsc{Theorem~\ref{thm:infismin}. (Infimax or Minimax)}
\,\, \emph{The infimum $\IM(\bal)$ of $\cM(\bal)$ is a minimum, i.e. is
  an element of~$\cM(\bal)$, if and only if $\bal$ is the only point
  of~$\Delta$ with its itinerary.  }

\bigskip

We say that~$\bal$ is {\em regular} if it is the only point
of~$\Delta$ with its itinerary, and that it is {\em exceptional}
otherwise. Whether~$\bal$ is regular or exceptional appears to depend
on the growth rate of the itinerary of~$\bal$ in a delicate way: our
final result gives a flavour of this dependence.

\bigskip

\noindent\textsc{Theorem~\ref{thm:regular}. (Regular or exceptional)
} \,\, \emph{If the itinerary of~$\bal$ grows at most
  quadratically then~$\bal$ is regular; on the other hand, if it grows
  sufficiently fast then $\bal$ is exceptional.}

\bigskip

We now discuss the dynamical implications of these results in more
detail. Let~$X$ be a shift-invariant subset of~$\Sigma$. The vector
$\brho(w)\in\Delta$ of asymptotic proportions of the letters in an
element~$w$ of~$X$, if well-defined, is called the {\em rotation
  vector} of~$w$, and the collection of all of the rotation vectors of
elements of~$X$ is called the {\em rotation set} $\rho(X)$
of~$X$. This terminology is by analogy with manifold dynamics: in
fact, in the authors' forthcoming paper {\em ``New rotation sets in a
  family of torus homeomorphisms''}, these symbolic rotation vectors are
related directly to rotation vectors for torus homeomorphisms, and the
techniques developed in this paper make it possible to provide a
detailed description of all of the rotation sets which arise in a
parameterised family of torus homeomorphisms.

When~$X$ is a subshift of finite type, a theorem of Ziemian~\cite{Zie}
states that~$\rho(X)$ is a convex set with finitely many extreme
points, given by the rotation vectors of the minimal loops of the
transition diagram. While this result is useful, subshifts of finite
type are rather special, and are often ill-suited to understand
dynamical behaviour in parameterised families, since Markov partitions
can change dramatically under small changes in the map. Here we
consider a broader class: in analogy with kneading theory and
beta-shifts, we consider subshifts of the form
\[X(v) = \{w\in\Sigma\,:\,\sigma^r(w)\le v \text{ for all }r\ge 0\},\]
where~$v\in\Sigma$. In fact, since the supremum of any
shift-invariant set is a maximal sequence, and since
$X(v) = X(\sup X(v))$, it suffices to consider the case
where~$v$ is maximal, which we henceforth assume.

Now if there is some $w\in \cM(\bal)$ with $w\le v$ then it is clear that
$\bal\in\rho(X(v))$, since if $w\le v$ and $w$ is maximal than $w\in
X(v)$. Recalling that $\IM(\bal)$ denotes the infimum of all of the
$w\in \cM(\bal)$, it follows that
\[v > \IM(\bal) \, \implies \, \bal\in\rho(X(v)).\]

Similarly, it can be shown (see Lemma~\ref{lem:lower-bound} below)
that if $w$ is {\em any} (not necessarily maximal) element of~$\Sigma$
with $\brho(w) = \bal$, then the supremum of the orbit of~$w$ is at
least $\IM(\bal)$. Therefore
\[v < \IM(\bal) \, \implies \, \bal\not\in\rho(X(v)).\]

Whether or not $\bal\in\rho(X(v))$ when $v = \IM(\bal)$ depends on
whether or not $\IM(\bal)$ has rotation vector~$\bal$: that is, on
whether it is an $\bal$-minimax, or only an $\bal$-infimax. Therefore
the results of this paper make it possible to determine whether or not
$\bal\in\rho(X(v))$ by comparing $v$ with the single
sequence~$\IM(\bal)$. Moreover, since a consequence of the above
discussion is that $\rho(X(v))$ can only change as $v$ passes through
an element of the set $\IM = \{\IM(\bal)\,:\,\bal\in\Delta\}$ of
infimaxes, understanding how the structure of $\rho(X(v))$ changes as
$v$ increases is closely related to understanding the structure
of the set of initial segments of~$\IM$.

\bigskip

It is well known~\cite{Sturmian1,Sturmian2} that when $k=2$, all of
the infimaxes are minimaxes and are the Sturmian sequences studied by
Morse and Hedlund~\cite{morsehedlund1,morsehedlund2}. Thus the infimax
sequences with $k\ge 3$ letters can be seen as extensions of the two
letter Sturmians (however, in contrast to the Sturmian case, infimax
sequences when $k>2$ are very far from being balanced and are not, in
general, of Arnoux-Rauzy type~\cite{GBA}). The construction of infimax
sequences described here is reminiscent of the construction of
Sturmian sequences through their relationship with continued fraction
expansions. First there is a division-remainder procedure, similar to
the standard Euclidean algorithm, which produces a sequence $\bn$ of
non-negative integers, analogous to the partial quotients of a
continued fraction expansion (this sequence is the itinerary of the
orbit of~$\bal$ under $K\colon \Delta\to\Delta$ with respect to a
certain partition of~$\Delta$, just as the sequence of partial
quotients of the continued fraction expansion of~$\alpha\in(0,1)$ is
the itinerary of~$\alpha$ under the Gauss map).  Second, this
itinerary is used to construct a sequence of substitutions which are
applied successively to the single letter~$k$, producing a sequence of
words of increasing lengths, each of which is an initial subword of
the infimax. If~$\bal$ is a rational vector then the minimax sequence
is periodic, and is determined after finitely many steps of the
algorithm.

Section~\ref{sec:defns-statments} contains basic definitions and
precise statements of the theorems described above. Some preliminary
results are presented in Section~\ref{sec:preliminaries}, and a finite
version of the problem is then treated in Section~\ref{sec:finite}:
given non-negative integers $a_1,\ldots,a_k$, what is the smallest
maximal {\em word} which contains exactly $a_i$ occurrences of each
letter~$i$? The solution of this problem is required later in the
paper, and also introduces the main ideas in a more straightforward
context.

In Section~\ref{sec:infimax} we prove the validity of the algorithm
for determining infimax sequences, before finishing, in
Section~\ref{sec:minimax}, by considering the conditions under which
infimax sequences are minimaxes.

\medskip\medskip\medskip\medskip %FORMAT

\section{Definitions, notation, and statement of results}
\label{sec:defns-statments}

Let~$k\ge 2$ be the number of letters in our alphabet
$\cA=\{1,\ldots,k\}$. We fix~$k$ throughout, and suppress the
dependence of objects on it, except in Remark~\ref{rmk:reduce-k} and
in the final part of the proof of Theorem~\ref{thm:infismin}.

Denote by~$\Sigma$ the space~$\cA^\N$ of sequences with entries
in~$\cA$: we consider~$0$ to be a natural number, so that elements~$w$
of~$\Sigma$ are indexed as $w=(w_r)_{r\ge 0}$. Order~$\Sigma$
lexicographically, and endow it with the product topology (where~$\cA$
is discrete).

Similarly, denote by~$\cA^*$ the set of non-trivial finite words over
the alphabet~$\cA$, ordered lexicographically with the convention that
any proper initial subword of $W\in\cA^*$ is greater than~$W$ (this
convention is simply to ensure that $\cA^*$ is totally ordered, and
does not affect any of the results of the
paper). Given~$W\in\cA^*$ and $i\in\cA$, write $|W|\ge 1$ for the
length of~$W$, and $|W|_i\ge 0$ for the number of occurences of the
letter~$i$ in~$W$.

If~$V,W\in\cA^*$, denote by $VW$ the concatenation of~$V$ and~$W$,
by $\overline{W} = WWWW\ldots$ the element of~$\Sigma$ given by
infinite repetition of~$W$, and by $V\overline{W}$ the element
$VWWWW\ldots$ of~$\Sigma$. An element of~$\Sigma$ of the
form~$\overline{W}$ is said to be {\em periodic}. Given $W\in\cA^*$
and $n\ge 0$, denote $W^n=WW\ldots W$ the $n$-fold repetition
of~$W$, an element of~$\cA^*$ provided that~$n>0$: if $n=0$ then
$W^n$ denotes the empty word, which will be used only when
concatenated with elements of~$\cA^*$.

If $w\in\Sigma$ and $r\ge 1$ is an integer, write $w^{(r)} =
w_0w_1\ldots w_{r-1}$, the element of~$\cA^*$ formed by the first~$r$
letters of~$w$.

The shift map $\sigma\colon\Sigma\to\Sigma$ is defined by $\sigma(w)_r
= w_{r+1}$. An element $w$ of~$\Sigma$ is said to be {\em maximal} if
it is the maximum element of its $\sigma$-orbit: that is, if
$\sigma^r(w)\le w$ for all~$r\ge 0$. We write~$\cM\subset\Sigma$ for the
set of maximal elements. Observe that~$\cM$ is a closed subset
of~$\Sigma$, for if $w\in\Sigma$ is not maximal then there is
some~$r\ge 0$ with $\sigma^r(w)>w$, and it follows that
$\sigma^r(w')>w'$ for all~$w'\in\Sigma$ sufficiently close to~$w$.

Given~$W\in\cA^*$, write $\brho(W)\in\Q^k$ for the vector whose
$i^\text{th}$ component is the proportion of the letter~$i$ in~$W$:
that is, $\brho(W)_i = |W|_i/|W|$. Let
\[\Delta = \left\{\bal\in\R_{\ge 0}^k\,:\,\alpha_k>0,\,\sum \alpha_i
= 1\right\},\] the simplex which contains these rational vectors,
with the face $\alpha_k=0$ removed, equipped with the maximum metric
$d_\infty$. Removing the face~$\alpha_k=0$ makes the statements of the
results of the paper cleaner, and clearly if~$\alpha_k=0$ then the
problem reduces to one with a smaller value of~$k$.

Given~$\bal\in\Delta$, denote by~$\cR(\bal)$ the set of elements
of~$\Sigma$ with asymptotic proportions of letters~$\bal$:
\[
\cR(\bal) = \left\{ w\in\Sigma\,:\, \brho\left(w^{(r)}\right) \to
\bal \text{ as }r\to\infty\right\} \subset\Sigma.
\]

\begin{remark}
\label{rmk:head-and-tail}
$\cR(\bal)$ is not closed in~$\Sigma$. For example, when~$k=2$ the
sequence $2^r\,\overline{21}$ is an element of $\cR(1/2,1/2)$ for
all~$r\ge 0$, but $2^r\,\overline{21}\to
\overline{2}\not\in\cR(1/2,1/2)$ as $r\to\infty$. This is a
consequence of the more general observation that the asymptotic
proportions of elements of~$\Sigma$, which depend on their tails, do
not interact well with the order and topology on~$\Sigma$, which are
defined using the heads of its elements.
\end{remark}

We define also the set of maximal sequences with proportions~$\bal$,
\[
\cM(\bal) = \cM\cap \cR(\bal).\]
Following on from Remark~\ref{rmk:head-and-tail}, observe that it is
easy to construct elements of~$\cM(\bal)$. Provided that~$\bal
\not=(0,0,\ldots,0,1)$ then there are elements of~$\cR(\bal)$ for which
there is an upper bound~$N$ on the number of consecutive occurences of the
letter~$k$, and prepending $k^{N+1}$ to such an element yields an
element of~$\cM(\bal)$. On the other hand, if $\bal=(0,0,\ldots,0,1)$
then $\overline{k}\in\cM(\bal)$. In particular, since every non-empty
subset of~$\Sigma$ has an infimum, we can define the {\em
  $\bal$-infimax} sequence~$\IM(\bal)$ by
\[\IM(\bal) = \inf \cM(\bal).\]

$\IM(\bal)$ is necessarily an element of~$\cM$, but need not be an
element of~$\cR(\bal)$, which is not closed in~$\Sigma$. In the
case that it is (and so is an element of~$\cM(\bal)$), we call it the
{\em $\bal$-minimax sequence}.

\medskip\medskip

Having introduced the basic objects of study, we now turn to the
algorithm for constructing $\IM(\bal)$, which is given in terms of the
itinerary of~$\bal$ under a certain dynamical system $K\colon
\Delta\to\Delta$, defined piecewise on the subsets
\[\Delta_n = \left\{\bal\in\Delta\,:\,\left \lfloor
\frac{\alpha_1}{\alpha_k} \right \rfloor = n \right\} \subset \Delta
\qquad (n\in\N),\]
where $\lfloor x\rfloor$ denotes the integer part of~$x$.
First, let $K_n\colon\Delta_n\to\Delta$ be given by 
\begin{equation}
\label{eq:Kn}
K_n(\bal) = \left(
\frac{\alpha_2}{1-\alpha_1}, \,
\frac{\alpha_3}{1-\alpha_1}, \,
\ldots, \,
\frac{\alpha_{k-1}}{1-\alpha_1}, \,\,\,
\frac{\alpha_1-n\alpha_k}{1-\alpha_1}, \,
\frac{(n+1)\alpha_k - \alpha_1}{1-\alpha_1}
\right).
\end{equation}
Each~$K_n$ is a {\em projectivity}: an embedding induced on a subset
of~$\R^k$ by the action of an element of $\operatorname{GL}_{k+1}(\R)$ on projective
coordinates in $\RP^k$. As such, it sends convex sets to convex
sets. Its inverse $K_n^{-1}\colon \Delta\to\Delta_n$ is given by
\begin{equation}
\label{eq:KnI}
K_n^{-1}(\bal) = 
\left(
\frac{(n+1)\alpha_{k-1}+n\alpha_k}{D},\,\,
\frac{\alpha_1}{D},\,\,
\frac{\alpha_2}{D},\,\,
\ldots,\,
\frac{\alpha_{k-2}}{D},\,\,
\frac{\alpha_{k-1}+\alpha_k}{D}
\right),
\end{equation}
where~$D=(n+1)\alpha_{k-1}+n\alpha_k+1$.

Let~$J\colon\Delta\to\N$ be given by $J(\bal) =
\lfloor\alpha_1/\alpha_k\rfloor$, so that
$\bal\in\Delta_{J(\bal)}$. Then define $K\colon\Delta\to\Delta$
by \[K(\bal) = K_{J(\bal)}(\bal),\] which is a multi-dimensional
continued fraction map~\cite{Schweiger}.
Associated to $K$ is an itinerary map $\Phi\colon\Delta\to\N^\N$
defined by
\[\Phi(\bal)_r = J(K^r(\bal)) \qquad(r \in\N).\]

We shall see that the infimax sequence $\IM(\bal)$ is obtained from a
sequence of substitutions associated with $\Phi(\bal)$. Recall that a
{\em substitution on $\cA$} is a map
$\Lambda\colon\cA\to\cA^*$. Overloading notation, this induces maps
$\Lambda\colon\cA^*\to\cA^*$ and $\Lambda\colon\Sigma\to\Sigma$ which
replace each letter of the input sequence with its image:
$\Lambda(w_0w_1w_2\ldots) =
\Lambda(w_0)\Lambda(w_1)\Lambda(w_2)\ldots$. Define
substitutions~$\Lambda_n$ for each $n\in\N$ by
\begin{equation}
\label{eq:lambda_n}
\Lambda_n\colon \qquad
\left\{
\begin{array}{lll}
  i & \mapsto & (i+1) \qquad\qquad \text{ if }1\le i\le k-2,\\
  (k-1) & \mapsto & k\, 1^{n+1}, \qquad\qquad\text{and} \\
  k & \mapsto & k\, 1^n. 
\end{array}
\right.
\end{equation}

Observe that the expression~(\ref{eq:KnI}) for $K_n^{-1}(\bal)$
results precisely from translating~(\ref{eq:lambda_n}) in such a way
as to give the proportions of each letter in~$\Lambda_n(w)$ in
terms of the proportions in~$w$, that is,
\[w\in \cR(\bal) \iff \Lambda_n(w) \in \cR(K_n^{-1}(\bal)).\]

Given $\bn\in\N^\N$, define substitutions $\Lambda_{\bn,r}$ for
each~$r\in\N$ by
\[\Lambda_{\bn,r} =
\Lambda_{n_0}\circ\Lambda_{n_1}\circ\cdots\circ\Lambda_{n_r}.\]
Then define a map $\Wmap\colon \N^\N\to\Sigma$ by
\[\Wmap(\bn) = \lim_{r\to\infty} \Lambda_{\bn,r}\left(\overline{k}\right) =
\lim_{r\to\infty}\overline{\Lambda_{\bn,r}(k)},\] 
where in the first definition $\Lambda_{\bn,r}$ is regarded as a map
$\Sigma\to\Sigma$, and in the second as a map $\cA^*\to\cA^*$. The
limit exists since $\Lambda_{n_{r+1}}(k)$ begins with the letter~$k$,
and hence $\Lambda_{\bn,r}(k)$ is an initial subword of
$\Lambda_{\bn,r+1}(k)$ for all~$r$.

The first main theorem of the paper states that, for every
$\bal\in\Delta$, the corresponding infimax sequence is given
by~$\Wmap(\Phi(\bal))$.

\bigskip

\noindent\textsc{Theorem~\ref{thm:infimax}.} \,\, \emph{Let
  $\bal\in\Delta$. Then $\IM(\bal) = \Wmap(\Phi(\bal))$.}

\bigskip

The question of whether or not the infimax sequence is a minimax
(that is, of whether or not an $\bal$-minimax exists) is
answered by the following result:

\bigskip

\noindent\textsc{Theorem~\ref{thm:infismin}.} \,\, \emph{Let
  $\bal\in\Delta$. Then
\begin{enumerate}[a)]
\item $\Phi^{-1}(\Phi(\bal))\subset\Delta$ is a $d$-dimensional
  simplex for some~$d$ with $0\le d\le k-2$. 
\item $\IM(\bal)$ is the minimum of $\cM(\bal)$ if and only if
  $\Phi^{-1}(\Phi(\bal))$ is a point.
\end{enumerate}
}

\bigskip

There is therefore a fundamental distinction between {\em regular}
elements $\bal$ of $\Delta$, for which $\Phi^{-1}(\Phi(\bal))$ is a
point, and {\em exceptional} elements for which this is not the
case. That both possibilities occur is the content of the following
theorem.

\bigskip

\noindent\textsc{Theorem~\ref{thm:regular}.} \,\,
\emph{Let~$\bal\in\Delta$ and $\bn = \Phi(\bal)$.
\begin{enumerate}[a)]
\item If there is some~$C$ such that $0<n_r \le Cr^2$ for all~$r$, then
  $\bal$ is regular.
\item If $k\ge 3$ and $n_r \ge 2^{r+2}\prod_{i=0}^{r-1}(n_i+2)$ for
  all~$r\ge 1$, then 
  $\Phi^{-1}(\Phi(\bal))$ is a simplex of dimension~$k-2$, so that
  $\bal$ is exceptional.
\end{enumerate}}

\bigskip

Notice that if~$k=2$ then Theorem~\ref{thm:infismin}~a) gives that
every~$\bal\in\Delta$ is regular.

The growth condition of Theorem~\ref{thm:regular}~b) is designed for
ease of proof and can be improved without difficulty. Providing a
precise characterisation of the set of regular~$\bal$ when $k\ge 3$,
by contrast, appears to be a challenging problem.

\bigskip\bigskip

The substitutions~$\Lambda_n$ which play a central r\^ole here appear
in a different context in papers of Bruin and Troubetzkoy~\cite{BT}
and Bruin~\cite{Bruin} (dealing respectively with the case~$k=3$ and
the case~$k\ge 3$). These papers are concerned with a certain class of
{\em interval translation mappings} (which are defined similarly to
interval exchange mappings except that the images of the monotone
pieces can overlap). The most interesting case is when the maps are of
infinite type, which means that the attractor is a Cantor set. An
interval translation mapping with $k$ monotone pieces which is of
infinite type can be renormalized infinitely often, with each
renormalization being described by a substitution on the space of
$k$-symbol itineraries. The dynamics on the attractor is therefore
given by the subshift generated by the sequence of substitutions
corresponding to the sequence of renormalizations.

It turns out that the substitutions arising from renormalization of
$k$-piece interval translation mappings in the class considered by
Bruin and Troubetzkoy are exactly the substitutions~$\Lambda_n$
of~(\ref{eq:lambda_n}). Their results provide extensions of some of
the results of this paper, particularly in the case~$k=3$: see
Remarks~\ref{rmk:infismin}~b) and~\ref{rmk:regular}~a).

\section{Preliminaries}
\label{sec:preliminaries}
In this section we state some basic facts about the maps defined in
Section~\ref{sec:defns-statments}. The proofs are routine, and could
be omitted on first reading. The crucial result for what follows is
Corollary~\ref{cor:comp-LSC}, which asserts that the map
$\Wmap\circ\Phi$ is lower semi-continuous.

\begin{lemma}
\label{lem:actionLambda}
Let~$n\in\N$. Then the substitution $\Lambda_n\colon\cA^*\to\cA^*$ is
strictly order-preserving. Similarly $\Lambda_n\colon\Sigma\to\Sigma$ is
strictly order-preserving, with \mbox{$\Lambda_n(\cM) \subseteq \cM$}.
\end{lemma}

\begin{proof}
To show that~$\Lambda_n\colon\cA^*\to\cA^*$ is strictly order-preserving,
suppose that $V,W\in\cA^*$ with \mbox{$V<W$}. Then either $W$ is a proper
initial subword of~$V$, in which case $\Lambda_n(W)$ is a proper
initial subword of $\Lambda_n(V)$, so that $\Lambda_n(V)<\Lambda_n(W)$
as required; or there is some~$R\ge 0$ with $V_r=W_r$ for $0\le r<R$
and $V_R<W_R$. If $V_R\le {k-2}$ then it is obvious that $\Lambda_n(V)
< \Lambda_n(W)$. On the other hand, if $V_R=k-1$ and $W_R=k$, then
$\Lambda_n(V) = \Lambda_n(V_0\ldots V_{R-1})k1^{n+1}\ldots$, and
$\Lambda_n(W) = \Lambda_n(V_0\ldots V_{R-1})k1^n\ldots$. If $W$ has
length~$R+1$ then $\Lambda_n(W)$ is a proper initial subword of
$\Lambda_n(V)$, so that $\Lambda_n(V)<\Lambda_n(W)$; while if $W$ has
length greater than~$R+1$, then the letter following
$\Lambda_n(V_0\ldots V_{R-1})k1^n$ in $\Lambda_n(W)$, being the first
letter in the $\Lambda_n$-image of a letter, is not~$1$, so again
$\Lambda_n(V)<\Lambda_n(W)$ as required.

The proof that $\Lambda_n\colon\Sigma\to\Sigma$ is strictly
order-preserving is similar but simpler, since there is no longer any
need to worry about the ends of the words.

To show that $\Lambda_n(\cM)\subseteq \cM$, let~$w\in \cM$. Consider $w_0$,
the first, and hence largest, letter in~$w$. If $w_0 < k-1$ then
$\Lambda_n(w_r) = w_r + 1$ for all~$r$, and it is clear that
$\Lambda_n(w)\in \cM$. Assume therefore that $w_0\ge k-1$, so that
$\Lambda_n(w)$ begins with the letter~$k$. Suppose for a contradiction
that $\Lambda_n(w)$ is not maximal, so that $\Lambda_n(w) = Vv$ for
some $V\in\cA^*$ and $v\in\Sigma$ with $v>Vv$. Since~$V_0=k$ we must
have $v_0=k$. Since~$k$ can only occur as the first letter in the
$\Lambda_n$-image of a letter, it follows that $w = Uu$ with
$\Lambda_n(U) = V$ and $\Lambda_n(u)=v$. Since $w$ is maximal we have
$u\le Uu$, and since $\Lambda_n$ is order-preserving we have $v =
\Lambda_n(u)\le \Lambda_n(Uu) = Vv$, which is the required
contradiction.
\end{proof}

The following lemma is an immediate consequence of the
definition~(\ref{eq:lambda_n}) of the substitutions~$\Lambda_n$.

\begin{lemma}
\label{lem:eventually-k}
Let $n_0,n_1,\ldots,n_{k-2}$ be any natural numbers. Then
$\Lambda_{n_0}\circ\Lambda_{n_1}\circ \cdots
\circ\Lambda_{n_{k-2}}(W)$ has initial letter~$k$ for all~$W\in\cA^*$.
\end{lemma}

\begin{proof}
If $W_0 = i<k$, then $(\Lambda_n(W))_0=i+1$ for all~$n\in\N$; while if
$W_0=k$ then $(\Lambda_n(W))_0=k$ also.
\end{proof}

Endow~$\N^\N$ with the product topology, and order it {\em
  reverse lexicographically}: that is, lexicographically with the
convention that $0>1>2>3>\cdots$. This convention is to ensure that
$\Wmap\colon \N^\N\to\Sigma$ is order-preserving.

\begin{lemma}
\label{lem:props_w}
$\Wmap\colon \N^\N\to \Sigma$ is continuous and strictly order-preserving, with
image contained in~$\cM$.
\end{lemma}

\begin{proof}
Let~$\bn\in\N^\N$.

To show that $\Wmap(\bn)\in \cM$, observe that
$\Lambda_{\bn,r}\left(\overline{k}\right)\in \cM$ for each~$r$ by
Lemma~\ref{lem:actionLambda}. The result follows since~$\cM$ is closed
in~$\Sigma$.

To show that~$\Wmap$ is continuous at~$\bn$, observe that since
$\Lambda_n(k)=k1^n$, the word $L_{\bn,r}:=\Lambda_{\bn,r}(k)$ has
length at least $1+\sum_{s=0}^r n_s$. Therefore if $\bm\in\N^\N$
satisfies $\bm^{(r+1)}=\bn^{(r+1)}$, then $\Wmap(\bm)$ and $\Wmap(\bn)$
agree to at least $1+\sum_{s=0}^r n_s$ letters. This establishes that
$\Wmap$ is continuous at~$\bn$ provided that $n_s\not=0$ for
arbitrarily large~$s$.

To show continuity at~$\bn$ in the case where $\bn = n_0\ldots
n_{r-1}\overline{0}$ for some $r\ge 0$, observe that, for $R\ge k-1$,
\[\Lambda_0^R(k1) = \Lambda_0^{R-1}(k2) = \cdots =
\Lambda_0^{R-k+2}(k\,(k-1)) = \Lambda_0^{R-k+1}(k\,k\,1),\] and,
repeating the argument, $\Lambda_0^R(k1)$ has initial subword
$k^{1+\lfloor R/(k-1)\rfloor}$. Now if $\bm\not=\bn$ is very
close to~$\bn$, then $\bm = n_0\ldots n_{r-1}\,0^R\,m_{r+R}\ldots$, where
$R$ is very large and $m_{r+R}>0$. It follows that
\[L_{\bm,r+R} = \Lambda_{\bn,r-1}(\Lambda_0^R(k1\ldots)) =
\Lambda_{\bn,r-1}(k^{1+\lfloor R/(k-1)\rfloor}\,\ldots)\]
agrees with $S(\bn)=\Lambda_{\bn,r-1}\left(\overline{k}\right)$ to at
least $1 + \lfloor R/(k-1)\rfloor$ letters, establishing continuity
at~$\bn$ as required.

To show that $\Wmap$ is strictly order-preserving, let $\bm\in\N^\N$
with $\bm < \bn$, so that there is some $r\in\N$ with
$\bm^{(r)}=\bn^{(r)}$ but $m_r>n_r$ (since~$\N^\N$ is ordered
reverse lexicographically). Then
$\Lambda_{\bn,r-1}(\Lambda_{n_r}(k\ell))$ is an initial subword of
$\Wmap(\bn)$ for some letter~$\ell\in\cA$, while
$\Lambda_{\bn,r-1}(\Lambda_{m_r}(k))$ is an initial subword of
$\Wmap(\bm)$. Now $\Lambda_{m_r}(k) = k1^{m_r} <
k1^{n_r}\Lambda_{n_r}(\ell) = \Lambda_{n_r}(k\ell)$ since $m_r>n_r$,
so that $\Wmap(\bm)<\Wmap(\bn)$ by Lemma~\ref{lem:actionLambda} as
required.
\end{proof}

Using the definitions of the product topology on~$\Sigma$ and the
lexicographical order on~$\cA^*$, the standard definition of lower
semi-continuity for functions from a metric space~$X$ into~$\Sigma$
can be phrased as follows: $f\colon X\to\Sigma$ is lower
semi-continuous at~$x\in X$ if
\[
\forall R\in\N,\,\exists\epsilon>0,\quad d(x,y)<\epsilon \implies
f(y)^{(R)} \ge f(x)^{(R)}. 
\]
Similarly, $f\colon X\to\N^\N$ is lower semi-continuous at~$x$ if the
same condition holds, bearing in mind that the $\ge$ should be
interpreted reverse lexicographically. 

Although the itinerary
map~$\Phi\colon\Delta\to\N^\N$ is discontinuous at all preimages
under~$K$ of the discontinuity set of~$K$, it is everywhere lower
semi-continuous:

\begin{lemma}
\label{lem:LSC}
$\Phi\colon\Delta\to\N^\N$ is lower semi-continuous.
\end{lemma}

\begin{proof}
We need to show that for all~$R\in\N$ and all $\bal\in\Delta$,
there is an~$\epsilon>0$ such that if $d_\infty(\bal,\bbeta)<\epsilon$
then $\Phi(\bbeta)^{(R)} \ge \Phi(\bal)^{(R)}$. The proof is by
induction on~$R$.

 For the case~$R=0$, observe that for all~$\bal\in\Delta$ there is
 some~$\epsilon>0$ such that if $d_\infty(\bal,\bbeta)<\epsilon$ then
 $\bbeta\in \Delta_{J(\bal)}\cup\Delta_{J(\bal)-1}$, so that
 $J(\bbeta)\le J(\bal)$ and hence $\Phi(\bbeta)^{(0)} \ge
 \Phi(\bal)^{(0)}$ as required.

If~$R>0$, then for each~$\bal\in\Delta$ there is, by the inductive
hypothesis, some~$\delta>0$ such that if
$d_\infty(K(\bal),\bgamma)<\delta$ then $\Phi(\bgamma)^{(R-1)} \ge
\Phi(K(\bal))^{(R-1)}$. Then, by continuity of $K_{J(\bal)}$, there is
some \mbox{$\epsilon>0$} such that if $d_\infty(\bal,\bbeta)<\epsilon$
then either $J(\bbeta)=J(\bal)$ and $d_\infty(K(\bal),
K(\bbeta))<\delta$; or \mbox{$J(\bbeta)<J(\bal)$}. In either case,
$\Phi(\bbeta)^{(R)} \ge \Phi(\bal)^{(R)}$ as required.
\end{proof}

Combining Lemma~\ref{lem:props_w} and Lemma~\ref{lem:LSC} gives

\begin{corollary}
\label{cor:comp-LSC}
$\Wmap\circ\Phi\colon\Delta\to \cM$ is lower semi-continuous.
\end{corollary}

The next lemma and remark describe the case in which one of the components
of~$\bal$ is zero, so that the problem can be reduced to one over a
smaller alphabet.

\begin{lemma}
\label{lem:zero-component}
Let~$\bal\in\Delta$ have itinerary $\bn = \Phi(\bal)$, and let $1\le
i\le k-1$. Then $\alpha_i = 0$ if and only if $n_r = 0$ for all $r
\equiv i-1 \bmod k-1$.
\end{lemma}

\begin{proof}
By~(\ref{eq:Kn}), if $\alpha_1=0$ then
$K(\bal)_{k-1}=K_0(\bal)_{k-1}=0$, while if $\alpha_i=0$ for some~$i$
with $2\le i\le k-1$ then $K(\bal)_{i-1} = 0$. Since $\Phi(\bal)_0 =
0$ whenever $\alpha_1 = 0$, it follows immediately that if
$\alpha_i=0$ then $n_r=0$ for all $r \equiv i-1 \bmod k-1$.

\medskip

For the converse observe first, by a straightforward induction
on~$i$, that if $0\le i \le k-2$ and if $n_0=0$, then
\[
K_{n_i}\circ K_{n_{i-1}} \circ \cdots \circ K_{n_0}(\bal)_j = 
\begin{cases}
{\alpha_{j+i+1}}\big/\left(1-\sum_{\ell=1}^{i+1}\alpha_\ell\right) & \text{ if
}1\le j \le k-i-2,\\
{\alpha_1}\big/\left(1-\sum_{\ell=1}^{i+1}\alpha_\ell\right) & \text{ if
} j = k-i-1,
\end{cases}
\]
independently of $n_1,\ldots,n_i$. The case $i=k-2$ gives
\[
K_{n_{k-2}}\circ K_{n_{k-1}} \circ \cdots \circ K_{n_0}(\bal)_1 =
\frac{\alpha_1}{\alpha_k} \ge \frac{\alpha_1}{1-\alpha_1},
\]
provided only that $n_0 = 0$. Now if $n_r = 0$ for all~$r\equiv 0
\bmod k-1$ then repeated application of this inequality gives
$K^{s(k-1)}(\bal)_1 \ge \alpha_1/(1-\alpha_1)^s$ for all~$s\ge 0$, and
since $K^{s(k-1)}(\bal)_1 < 1$ for all~$s$ it follows that
$\alpha_1=0$, establishing the converse in the case~$i=1$.

\medskip

The statement for arbitrary~$i\le k-1$ follows. For if $n_r = 0$ for
all~$r\equiv i-1 \bmod k-1$, then $\bbeta = K^{i-1}(\bal)$ has
itinerary $\bm = \Phi(\bbeta)$ satisfying $m_r=0$ for all $r\equiv 0
\bmod k-1$. Therefore $\beta_1=0$, and hence $\alpha_i=0$ by $i-1$
applications of~(\ref{eq:KnI}).
\end{proof}

\begin{corollary}
Let $\bal\in\Delta$ and $\bn=\Phi(\bal)$. Then the following are
equivalent:
\begin{enumerate}[a)]
\item For all $r\ge 0$, no component of $K^r(\bal)$ is zero; and
\item For all $r\ge 0$, there is some $s\ge 0$ with
  $n_{r+s(k-1)}\not=0$. 
\end{enumerate}
\end{corollary}

\begin{remark}
\label{rmk:reduce-k}
This remark relates the $K$-orbit of~$\bal$ when some~$\alpha_i=0$ to
the $K$-orbit of the point~$\bbeta$ obtained by deleting the
$i^\text{th}$ component of $\bal$. We therefore include the value
of~$k$ in our notation, writing $\Delta^k$ instead of~$\Delta$. For
each $i$ with $1\le i\le k-1$, write \mbox{$\Delta^{k,i} =
  \{\bal\in\Delta^k\,:\, \alpha_i=0\}$}, the $i^\text{th}$ face
of~$\Delta$, and let $\pi_i\colon\Delta^{k,i}\to \Delta^{k-1}$ be the
bijection which forgets~$\alpha_i$: that is, $\pi_i(\bal) =
(\alpha_1,\ldots,\alpha_{i-1}, \alpha_{i+1}, \ldots, \alpha_k)$. The
bijection~$\pi_i$ clearly also depends on~$k$, as do the maps~$K$, but
no confusion will arise from continuing to suppress this dependence.

Now if $k\ge 3$ and $\bal\in\Delta^{k,i}$, then it follows directly 
from~(\ref{eq:Kn}) that
\[
K(\bal) = 
\begin{cases}
\pi_{k-1}^{-1}\circ\pi_1(\bal) & \text{ if }i = 1,\\
\pi_{i-1}^{-1}\circ K\circ \pi_i(\bal) & \text{ if }2\le i\le k-1.
\end{cases}
\]
 In particular, the itinerary $\Phi(\pi_i(\bal))$ is obtained from
 $\Phi(\bal)$ by deleting the zeroes which occur at each position~$r
 \equiv i-1\bmod k-1$.
\end{remark}

As stated in Section~\ref{sec:defns-statments}, if $\bn\in\N^\N$ it is
not in general the case that there is only a single point of~$\Delta$
with itinerary $\bn$. However, it is a straightforward consequence of
Lemma~\ref{lem:zero-component} that $\Phi^{-1}(\bn)$ is a single point
for itineraries of the form $\bn=W\,\overline{0}$.

\begin{lemma}
\label{lem:itinerary0bar}
Let $\bn = n_0n_1\ldots n_{r-1}\overline{0}\in\N^\N$. Then there is a
unique $\bal\in\Delta$ with $\Phi(\bal) = \bn$, namely 
$
\bal = K_{n_0}^{-1}\circ K_{n_1}^{-1} \circ \cdots \circ K_{n_{r-1}}^{-1} (0,0,\ldots,0,1)
$.
\end{lemma}

\begin{proof}
$\Phi(\bal)=\bn$ if and only if $\bal = K_{n_0}^{-1}\circ K_{n_1}^{-1}
  \circ \cdots \circ K_{n_{r-1}}^{-1}(\bbeta)$ for some $\bbeta$ with
  itinerary~$\overline{0}$. But $\Phi(\bbeta) = \overline{0}$ if and
  only if $\bbeta = (0,0,\ldots,0,1)$ by
  Lemma~\ref{lem:zero-component}.
\end{proof}

In particular, if $\bal$ has an itinerary of this form then
$\bal\in\Q^k$. Theorem~\ref{thm:restatement} states that, conversely,
every element~$\bal$ of $\Delta\cap\Q^k$ has such an itinerary: that
is, that $K^r(\bal) = (0,0,\ldots,0,1)$ for some $r$.

\section{The finite version}
\label{sec:finite}
In this section we solve a finite version of the minimax problem,
which is a necessary precursor to our later results. The simplicity of
the solution makes it straightforward to understand the origin of the
maps~$K_n$ and the substitutions~$\Lambda_n$.

A word $W\in\cA^*$ is said to be maximal if $\overline{W}$ is a
maximal element of $\Sigma$ or, equivalently, if $W=UV \implies W\ge
VU$, i.e. $W$ is at least as large as all of its cyclic permutations.

Let $\DelF = \{\balF =
(\alphaF_1,\ldots,\alphaF_k)\in\N^k\,:\,\alphaF_k>0\}$, the discrete
analogue of the space~$\Delta$. Continuing the analogy, we write for
each $\balF\in\DelF$
\begin{eqnarray*}
\RF(\balF) &=& \{W\in\cA^*\,:\,|W|_i = \alphaF_i \text{ for }1\le i\le
k\}\qquad\text{(a finite set)},\\
\MF(\balF) &=& \{W\in \RF(\balF)\,:\, W \text{ is
  maximal}\},\qquad\text{and}\\
\IF(\balF) &=& \min \MF(\balF).
\end{eqnarray*}

\begin{remark}
\label{rmk:obvious}
An obvious comment, which is nevertheless important for the proof of
Theorem~\ref{thm:finite} below, is that every $W\in \RF(\balF)$ has a
cyclic permutation which belongs to $\MF(\balF)$.
\end{remark}

For each $n\in\N$ write $\DelF_n = \{\balF\in\DelF\,:\,n\,\alphaF_k\le
\alphaF_1<(n+1)\,\alphaF_k\}$, and define a bijection
$\KF_n\colon\DelF_n\to\DelF$ by

\[
\KF_n(\balF) = (\alphaF_2,\alphaF_3, \ldots,
\alphaF_{k-1}, \alphaF_1-n\alphaF_k, (n+1)\alphaF_k-\alphaF_1),
\]
whose inverse $\KF_n^{-1}\colon \DelF\to\DelF_n$, the Abelianization
of the substitution~$\Lambda_n$, is given by
\[
\KF_n^{-1}(\balF) = ((n+1)\alphaF_{k-1}+n\alphaF_k,
\alphaF_1,\alphaF_2, \ldots, \alphaF_{k-2}, \alphaF_{k-1}+\alphaF_k).
\]

\begin{lemma}
\label{lem:finiteLambda}
Let $\balF\in\DelF$. Then the set $\Lambda_n^{-1}(\MF(\balF))$ of words
whose image under~$\Lambda_n$ lies in $\MF(\balF)$ is exactly
$\MF(\KF_n(\balF))$.
\end{lemma}

\begin{proof}
To show that $\Lambda_n^{-1}(\MF(\balF)) \subset \MF(\KF_n(\balF))$,
let $W\in\cA^*$ with $\Lambda_n(W)\in \MF(\balF)$. Then $W\in
\RF(\KF_n(\balF))$ by comparison of the right-hand side
of~(\ref{eq:lambda_n}) with the formula for~$\KF_n^{-1}$. Moreover,
$W$ is maximal: for if $W=UV < VU$ then $\Lambda_n(W) =
\Lambda_n(U)\Lambda_n(V) < \Lambda_n(V)\Lambda_n(U)$ by
Lemma~\ref{lem:actionLambda}, contradicting the maximality of
$\Lambda_n(W)$.

To show that $\Lambda_n(W)\in \MF(\balF)$ for all $W\in
\MF(\KF_n(\balF))$, it follows as above that $\Lambda_n(W) \in
\RF(\balF)$. That it is maximal follows from translating the statement
$\Lambda_n(\cM)\subseteq\cM$ (Lemma~\ref{lem:actionLambda}) into the
finite setting:
\[W\text{ maximal } \implies \overline{W}\in \cM \implies
\Lambda_n(\overline{W}) = \overline{\Lambda_n(W)} \in \cM \implies
\Lambda_n(W)\text{ maximal.}\]
\end{proof}

The following theorem gives the fundamental relationship between the
substitutions, the linear maps associated to the division-remainder
algorithm, and the minimax: the substitution $\Lambda_n$ sends the
minimax for $\balF$ to the minimax for $\KF_n^{-1}(\balF)$.

\newpage%FORMAT

\begin{theorem}
\label{thm:finite}
Let $\balF\in\DelF$. 
\begin{enumerate}[(a)]
\item If $\alphaF_i=0$ for all $i<k$ then $\IF(\balF)=k^{\alphaF_k}$.
\item Otherwise, $\IF(\balF) = \Lambda_n(\IF(\KF_n(\balF)))$, where
  $n=\lfloor \alphaF_1/\alphaF_k\rfloor$. 
\end{enumerate}
\end{theorem}

\begin{remarks}\mbox{}
\label{rmk:finite}
\begin{enumerate}[a)]
\item The theorem gives rise to a straightforward algorithm for
  calculating $\IF(\balF)$: the key point is that the sum of the
  entries of $\KF_n(\balF)$ is $\alphaF_1$ less than the sum of the
  entries of $\balF$; and if $\alphaF_i>0$ for any $i<k$, then
  $\alphaF_1$ will be positive after applying~$i-1$ terms of the
  appropriate sequence of $\KF_n$'s. Therefore repeatedly applying
  $\KF_{\lfloor \alphaF_1/\alphaF_k\rfloor}$ eventually yields an
  $\balF$ with $\alphaF_i=0$ for all~$i<k$. The hand implementation of
  this algorithm is illustrated in Examples~\ref{example:finite}
  below, while the statement of Theorem~\ref{thm:finite} translates
  directly into a recursive algorithm for computer implementation.
\item By linearity of the $\KF_n$ we have $\IF(N\balF) = \IF(\balF)^N$
  for each integer $N\ge 1$.
\end{enumerate}
\end{remarks}

\begin{examples}
\label{example:finite}
Let $k=3$ and $\balF=(24,3,14)$. We have
\[(24,3,14)\topL1(3,10,4)\topL0(10,3,1)\topL{10}(3,0,1)\topL3(0,0,1),\]
so that
\[
\IF(24,3,14) = \Lambda_1\Lambda_0\Lambda_{10}\Lambda_3(3) 
= \Lambda_1\Lambda_0\Lambda_{10}(31^3)
= \Lambda_1\Lambda_0(31^{10}2^3)
= \Lambda_1(3\,2^{10}\,(31)^3)
= 31\,(311)^{10}\,(312)^3.
\]
Notice that the intermediate words $31^3$, $31^{10}2^3$, and
$3\,2^{10}\,(31)^3$ are $\IF(3,0,1)$, $\IF(10,3,1)$, and $\IF(3,10,4)$
respectively. 

Similarly, if $k=4$ and $\balF = (2,3,1,3)$ then
\[(2,3,1,3)\topL0 (3,1,2,1)\topL3(1,2,0,1)\topL1 (2,0,0,1) \topL2
(0,0,0,1),\]
so that
\[
\IF(2,3,1,3) = \Lambda_0\Lambda_3\Lambda_1\Lambda_2(4)
 = \Lambda_0\Lambda_3\Lambda_1(411)
 = \Lambda_0\Lambda_3(4122)
 = \Lambda_0(4111233) = 422234141.
\]
\end{examples}

\begin{proof}[of Theorem~\ref{thm:finite}]
Statement~(a) is obvious, since $k^{\alphaF_k}$ is the unique element
of $\RF(0,0,\ldots,0,\alphaF_k)$. 

For~(b), it suffices to show that $\IF(\balF)$ is in the image of
$\Lambda_n$, where $n=\lfloor \alphaF_1/\alphaF_k\rfloor$: the result then
follows immediately from Lemmas~\ref{lem:finiteLambda}
and~\ref{lem:actionLambda}. 

Since $\alphaF_1\ge n\alphaF_k$ there are elements of $\RF(\balF)$, and
hence, by Remark~\ref{rmk:obvious}, of $\MF(\balF)$, in which every
occurence of the letter~$k$ is followed by the word $1^n$, and such
elements of~$\MF(\balF)$ are smaller than any element of~$\MF(\balF)$
which does not have this property. Therefore
\[\IF(\balF) = k1^nW_1\,k1^nW_2\,\cdots\,k1^nW_{\alphaF_k}\]
for some words $W_r$ which do not contain the letter~$k$. Moreover,
the letters must be arranged in ascending order in each~$W_r$: that
is,
\[W_r = 1^{n_{r,1}}2^{n_{r,2}}\cdots (k-1)^{n_{r,k-1}}\]
for each~$r$, where the $n_{r,s}$ are non-negative integers. For if
this were not the case, then replacing each~$W_r$ with a word in which
the same letters are arranged in ascending order would decrease every
cyclic permutation of~$\IF(\balF)$ starting with~$k$, so that there
would be an element of $\MF(\balF)$ smaller than $\IF(\balF)$.

To show that $\IF(\balF)$ is in the image of~$\Lambda_n$, it therefore
suffices to show that $n_{r,1}\le 1$ for all~$r$. Observe first that
$\sum_{r=1}^{\alphaF_k}n_{r,1}=\alphaF_1-n\alphaF_k < \alphaF_k$, so that
at least one $n_{r,1}$ is zero, and in particular $n_{1,1}=0$ by
maximality of~$\IF(\balF)$.

Suppose for a contradiction that $n_{s,1}\ge 2$ for some
least~$s$. Define words~$W_r'$ for $1\le r\le \alphaF_k$ by $W_{s-1}' =
1W_{s-1}$, $W_s = 1W_s'$, and $W_r'=W_r$ for $r\not=s-1, s$: that is,
push one of the $1$s from $W_s$ to $W_{s-1}$. Then taking an
appropriate cyclic permutation yields an element~$W'$ of~$\MF(\balF)$
given by
\[W' =
k1^nW_t'\,k1^nW_{t+1}'\,\cdots\,k1^nW_{\alphaF_k}'\,k1^nW_1'\,\cdots
k1^nW_{t-1}',\]
where~$t$ is not equal to $s$ since $W_s'$ starts with the letter~$1$
by choice of~$s$, but $W_t'$ does not start with the letter~$1$ by
maximality of $W'$. Now
\[W'\,<\,
k1^nW_t\,k1^nW_{t+1}\,\cdots\,k1^nW_{\alphaF_k}\,k1^nW_1\,\cdots
k1^nW_{t-1}\,\le\, \IF(\balF),\] where the first inequality is by
definition of the words $W_r'$ together with $t\not=s$, and the second
is by maximality of~$\IF(\balF)$.  This contradicts that~$\IF(\balF)$
is the minimum element of~$\MF(\balF)$, establishing that~$\IF(\balF)$
is in the image of~$\Lambda_n$ as required.
\end{proof}

To connect this result with the formalism used in the general case,
observe that
\begin{equation}
\label{eq:CD}
\begin{CD}
\DelF_n @>\KF_n>>  \DelF  \\
@VV\pi V     @VV\pi V\\
\Delta_n @>K_n>> \Delta
\end{CD}
\end{equation}
commutes, where $\pi\colon\DelF\to\Delta$ is defined by
$\pi(\balF)=\balF/\sum\alphaF_i$. Moreover, the functions $\KF_n$ can be
gathered into a single function $\KF\colon\DelF\to\DelF$ defined by
$\KF(\balF) = \KF_{\lfloor\alphaF_1/\alphaF_k\rfloor}(\balF)$, giving rise
to an itinerary map $\PhiF\colon\DelF\to\N^\N$ defined by
\[\PhiF(\balF)_r = n \iff \KF^r(\balF)\in\DelF_n.\]
Since $\KF^r(\balF) = (0,0,\ldots,0,1)$, a fixed point of~$\KF$, for
some~$r$, the itinerary $\PhiF(\balF)$ has only finitely many non-zero
entries.

The following is then a restatement of Theorem~\ref{thm:finite}. Note
that it does not claim to give the minimum element of $\cM(\bal)$ for
rational~$\bal$, but only the minimum {\em periodic} element: that
this is in fact the minimum of $\cM(\bal)$ will follow from
Theorem~\ref{thm:infimax} below.

\begin{theorem}
\label{thm:restatement}
Let~$\bal\in \Delta\cap\Q^k$. Then the itinerary of $\bal$ is of
the form \mbox{$\Phi(\bal) = n_0\,n_1\,\ldots\,n_{r-1}\,\overline{0}$}, and the
minimum periodic element~$\MP(\bal)$ of $\cM(\bal)$ is equal to $\Wmap(\Phi(\bal))$.
\end{theorem}

\begin{proof}
Let~$\balF\in\DelF$ be the smallest integer vector which is a
positive multiple of~$\bal$. Then any periodic element of~$\cM(\bal)$ is
of the form $\overline{W}$, where $W\in\MF(N\balF)$ for some $N\ge
1$. However $\IF(N\balF) = \IF(\balF)^N$ by
Remark~\ref{rmk:finite}b), so that the smallest periodic element
of~$\cM(\bal)$ is $\overline{\IF(\balF)}$.

It is immediate from~(\ref{eq:CD}) that $\Phi(\bal) = \PhiF(\balF)$,
so that in particular $\bn=\Phi(\bal)$ is of the given form. Then
\[\MP(\bal) = \overline{\IF(\balF)} = \overline{\Lambda_{\bn,r-1}(k)} =
\Wmap(\Phi(\bal))\]
as required.
\end{proof}

\begin{remark}
In the computer science and combinatorics of words literature, the
term {\em Lyndon words} is used for words that are minimal amongst
their cyclic permutations with respect to the lexicographic
order~\cite{BP,Lyndon}. Therefore maximal words are the same as Lyndon
words when the ordering of~$\cA$ is reversed, and the results of this
section can be rephrased as determining the largest Lyndon word with a
given number of each of the letters.
\end{remark}

\section{Proof of Theorem~\ref{thm:infimax}: \,\, $\IM(\bal) = \Wmap(\Phi(\bal))$}
\label{sec:infimax}
In this section we prove that the infimum $\IM(\bal)$ of
$\cM(\bal)$ is given by $\Wmap(\Phi(\bal))$. We show first
(Lemma~\ref{lem:lower-bound}) that $\Wmap(\Phi(\bal))$ is a lower bound of
$\cM(\bal)$, and then (Lemma~\ref{lem:closure}) that it lies in the
closure of $\cM(\bal)$.

That $\Wmap(\Phi(\bal))$ is a lower bound of~$\cM(\bal)$ is a special case
of a more general result. Given any $w\in\Sigma$, define $\sup w\in \cM$ by
\[\sup w = \sup_{r\ge 0}\sigma^r(w),\]
so that $w = \sup w$ if and only if $w\in
\cM$. Lemma~\ref{lem:lower-bound} below states that if $w\in \cR(\bal)$ then
$\Wmap(\Phi(\bal))\le \sup w$: in particular, if $w\in \cM(\bal)$ then
$\Wmap(\Phi(\bal))\le w$ as required.

The proof uses the finite version of the result as expressed by
Theorem~\ref{thm:restatement}, and we start with a lemma which
provides appropriate rational approximations to $\bal$ together
with corresponding periodic approximations to the supremum of an
element of~$\cR(\bal)$.

\begin{lemma}
\label{lem:sup-approx}
Let~$\bal\in\Delta$, $w\in \cR(\bal)$, $R\in\N$ and $\epsilon>0$. Then
there is some~$\bbeta\in\Delta\cap\Q^k$ and a periodic $v\in
\cM(\bbeta)$ such that $d_\infty(\bal,\bbeta)<\epsilon$ and $(\sup w)^{(R)}=v^{(R)}$.
\end{lemma}

\begin{proof}
Write $s = \sup w$. By definition of the supremum, there is some $r\ge
0$ such that $(\sigma^r(w))^{(R)} = s^{(R)}$. Since $\sigma^r(w)\in
\cR(\bal)$, there is an initial subword $s^{(R)}\,W$ of $\sigma^r(w)$
long enough that
\[d_\infty(\bal, \brho(s^{(R)}\,W\,1^R)) < \epsilon.\]
Let $U$ be the length~$R$ word with the property that
$(\sigma^r(w))^{(2R+|W|)}= s^{(R)}\,W\,U$.

Let~$v\in \cM$ be the maximal shift of the periodic sequence
$u=\overline{s^{(R)}\,W\,1^R}$. We shall show that $v^{(R)} = s^{(R)}$
which will establish the result, with $\bbeta =
\brho(s^{(R)}\,W\,1^R)$. 

 Since $s$, and hence~$u$, begins with the letter~$k$, $v^{(R)}$
is a subword of $s^{(R)}\,W\,1^R \le s^{(R)}\,W\,U$: but every
length~$R$ subword of $s^{(R)}\,W\,U$ is a subword of $w$, and hence
is less than or equal to $s^{(R)}$ by the definition of the
supremum. Therefore $v^{(R)}\le s^{(R)}$. On the other hand, however,
$v^{(R)} \ge u^{(R)} = s^{(R)}$, since $v\ge u$. This establishes the result.
\end{proof}

\begin{lemma}
\label{lem:lower-bound}
Let~$\bal\in\Delta$ and $w\in \cR(\bal)$. Then $\Wmap(\Phi(\bal)) \le
\sup w$. In particular, $\Wmap(\Phi(\bal))$ is a lower
bound of $\cM(\bal)$.
\end{lemma}
\begin{proof}
Write $s = \sup w$. To show that $\Wmap(\Phi(\bal))\le s$, it suffices
to show that \mbox{$\Wmap(\Phi(\bal))^{(R)} \le s^{(R)}$} for
every~$R\in\N$. Fix such an~$R$.

By the lower semi-continuity of $\Wmap\circ\Phi$
(Corollary~\ref{cor:comp-LSC}), there is some~$\epsilon>0$ such that
if $d_\infty(\bal, \bbeta)<\epsilon$ then $\Wmap(\Phi(\bal))^{(R)} \le
\Wmap(\Phi(\bbeta))^{(R)}$.

By Lemma~\ref{lem:sup-approx} there is some~$\bbeta$ with
$d_\infty(\bal,\bbeta)<\epsilon$ and some periodic $v\in \cM(\bbeta)$
with \mbox{$s^{(R)} = v^{(R)}$}. Theorem~\ref{thm:restatement} gives
$v\ge \Wmap(\Phi(\bbeta))$. Then
\[\Wmap(\Phi(\bal))^{(R)} \le \Wmap(\Phi(\bbeta))^{(R)} \le v^{(R)} =
s^{(R)}\]
as required.
\end{proof}

We now turn to proving that $\Wmap(\Phi(\bal)) \in
\overline{\cM(\bal)}$. To do this we need to construct elements
of~$\cM(\bal)$ which agree with $\Wmap(\Phi(\bal))$ on arbitrarily long
initial subwords, and the following straightforward lemma will be used
for this purpose.

\begin{lemma}
\label{lem:rational-approx}
Let~$\bal\in\Delta$, $R\in\N$, and $\epsilon>0$. Then there
is some $\bbeta\in \Delta\cap\Q^k$ such that $d_\infty(\bal,
\bbeta)<\epsilon$ and $\Phi(\bbeta)^{(R)} = \Phi(\bal)^{(R)}$.
\end{lemma}
\begin{proof}
The proof is by induction on~$R$, with the base case~$R=0$ being the
statement that rational elements are dense in~$\Delta$.

Suppose then that~$R>0$. Let~$n=J(\bal)$, so that
$\bal\in\Delta_n$. Recall that $K|_{\Delta_n} =
K_n\colon\Delta_n\to\Delta$ is a homeomorphism. By the inductive
hypothesis, there is a sequence $(\bgamma_i)$ in~$\Delta\cap\Q^k$
converging to $K(\bal)$ with
$\Phi(\bgamma_i)^{(R-1)} = \Phi(K(\bal))^{(R-1)}$ for
all~$i$. Let~$\bbeta = K_n^{-1}(\bgamma_i)$ for some~$i$ large enough
that $d_\infty(\bal, \bbeta)<\epsilon$.
\end{proof}

\begin{lemma}
\label{lem:closure}
Let~$\bal\in\Delta$. Then $\Wmap(\Phi(\bal))\in\overline{\cM(\bal)}$.
\end{lemma}
\begin{proof}
If $\bal\in\Delta\cap\Q^k$ then $\Wmap(\Phi(\bal))\in \cM(\bal)$ by
Theorem~\ref{thm:restatement}, so we assume that $\bal\not\in\Q^k$,
and in particular, by Lemma~\ref{lem:itinerary0bar}, that
$\bn=\Phi(\bal)$ has infinitely many non-zero entries.

It suffices to find, for each $R$, an element~$w$ of~$\cM(\bal)$ with
initial subword $\Lambda_{\bn,R}(k)$. We can assume that $n_{R+1}>0$,
since otherwise we increase~$R$ until this is the case.

Using Lemma~\ref{lem:rational-approx}, find for each~$r\ge 0$ an
element $\bbeta_r$ of $\Delta\cap\Q^k$ with
$d_\infty(\bal,\bbeta_r)<1/2^r$, whose itinerary 
\[\Phi(\bbeta_r) = \bn_r = n_{r,0} n_{r,1}\ldots
n_{r,L_r}\overline{0},\]
satisfies $n_{r,s} = n_s$ for $0\le s\le R+1$.

Set
$U=\Lambda_{n_0}\Lambda_{n_1}\cdots\Lambda_{n_R}\Lambda_{n_{R+1}-1}(k)$,
and $W_r = \Lambda_{n_{r,0}}\Lambda_{n_{r,1}}\cdots
\Lambda_{n_{r,L_r}}(k)$ for each~$r$, so that $\brho(W_r) = \bbeta_r$.
Write $L = |U|$ and $L_r = |U_r|$ for $r\ge 0$. Choose integers
$p_r\ge 1$ for $r\ge 0$ inductively to satisfy $\sum_{s=0}^r p_sL_s >
2^rL_{r+1}$. Finally, set
\begin{equation}
\label{eq:w}
w = U\,W_0^{p_0}W_1^{p_1}W_2^{p_2}\ldots.
\end{equation}

We will show that $w\in \cM(\bal)$, which will establish the result
since it has initial subword $\Lambda_{\bn,R}(k)$. To show that $w\in
\cR(\bal)$, let $I = \{(r,s)\,:\,r\in\N,\, 0\le s < p_r\}$ ordered
lexicographically, and define an increasing function $\ell\colon
I\to\N$ by $\ell(r,s) = L + sL_r + \sum_{t=0}^{r-1} p_tL_t$, the index
of the beginning of the $(s+1)^\text{th}$ subword $W_r$
in~(\ref{eq:w}). Now since $\brho(W_r)\to\bal$ as $r\to\infty$ we
have that for all~$\epsilon>0$ there is some $J$ such that
$d_\infty(\bal, \brho(w^{(\ell(r,s))}))<\epsilon$ for all $(r,s) > (J,0)$. On
the other hand, given any $t\ge\ell(1,0)$, we have $d_\infty(\brho(w^{(t)}),
\brho(w^{(\ell(r,s))})) < 1/2^r$, where $(r,s)$ is greatest with
$\ell(r,s)\le t$, by choice of the $p_r$. Therefore $d_\infty(\bal,
\brho(w^{(t)})) \to 0$ as $t\to\infty$ as required.

It remains to show that~$w$ is maximal. Now we can write
\[w = \Lambda_{n_0}\Lambda_{n_1}\cdots \Lambda_{n_R}
(\Lambda_{n_{R+1}-1}(k) \Lambda_{n_{R+1}}(u))\]
for some $u\in\Sigma$. However $\Lambda_{n_{R+1}-1}(k)
\Lambda_{n_{R+1}}(u) = k\,1^{n_{R+1}-1} \Lambda_{n_{R+1}}(u)$ is
maximal, since it has initial subword $k\,1^{n_{R+1}-1}$ followed by a
letter other than~$1$, whereas every letter~$k$ in
$\Lambda_{n_{R+1}}(u)$ is followed by at least $n_{R+1}$ consecutive
$1$s. Therefore $w$ is also maximal by Lemma~\ref{lem:actionLambda}.
\end{proof}

Combining Lemmas~\ref{lem:lower-bound} and~\ref{lem:closure}
gives the result we have been working towards.

\begin{theorem}
\label{thm:infimax}
Let $\bal\in\Delta$. Then $\IM(\bal) = \Wmap(\Phi(\bal))$.
\hspace*{\fill}\proofbox
\end{theorem}

\begin{remarks}\mbox{}
\label{rmk:primed-and-almost-periodic}
\begin{enumerate}[a)]
\item The proofs of Lemmas~\ref{lem:sup-approx}
  and~\ref{lem:lower-bound} only depend on being able to find
  arbitrarily long initial subwords~$W$ of $w\in \cR(\bal)$ with
  $\brho(W)$ arbitrarily close to $\bal$. It follows that the results
  of this section remain true if elements of $\cR(\bal)$ are only
  required to have subsequential limits~$\bal$, which is a common
  approach in the definition of rotation sets. To be precise, for
  each~$\bal\in\Delta$ write
\[\cR'(\bal) = \left\{ w\in\Sigma\,:\, \brho\left(w^{(r_i)}\right) \to
\bal \text{ for some }r_i\to\infty\right\} \subset\Sigma,
\]
and $\cM'(\bal) = \cM\cap \cR'(\bal)$. Then $\Wmap(\Phi(\bal))$ is the
infimum of $\cM'(\bal)$, and $\Wmap(\Phi(\bal))\le \sup w$ for all~$w\in
\cR'(\bal)$. 

\item The infimax sequences $\Wmap(\Phi(\bal))$ are {\em almost
  periodic}: for every initial subword $W$ of $\Wmap(\Phi(\bal))$, there
  is some~$N$ with the property that every length~$N$ subword of
  $\Wmap(\Phi(\bal))$ contains~$W$. As a consequence, the orbit
  closure
\[\Sigma_\bal = \overline{\{\sigma^r(\Wmap(\Phi(\bal)))\,:\,r\ge
  0\}}\]
is a minimal $\sigma$-invariant set. 

To show almost periodicity, assume that $\bal\not\in\Q^k$ (since
otherwise $\Wmap(\Phi(\bal))$ is periodic and therefore almost
periodic), and write $\bn = \Phi(\bal)$. Pick~$r$ large enough that
$\Lambda_{\bn,r}(k)$ has initial subword~$W$. Now
$\Lambda_{n_{r+1}}\circ\Lambda_{n_{r+2}}\circ \cdots \circ
\Lambda_{n_{r+k-1}} (i)$ has initial letter~$k$ for all $i$ with $1\le
i\le k$ by Lemma~\ref{lem:eventually-k}, so that $U_i := \Lambda_{\bn,
  r+k-1}(i)$ has initial subword~$W$ for each~$i$. However
$\Wmap(\Phi(\bal)) = \Lambda_{\bn, r+k-1}(u)$ for some~$u\in\Sigma$,
and is therefore a concatenation of the words~$U_i$. This establishes
the result, with $N=2\,\max_{1\le i\le k}|U_i|$.
\end{enumerate}

\end{remarks}

\section{Minimax sequences}
\label{sec:minimax}
In this section we address the question of when the infimum
$\IM(\bal)$ of $\cM(\bal)$ is a minimum. Since the set of maximal
elements is a closed subset of~$\Sigma$, $\IM(\bal)$ is necessarily
maximal, and the issue is whether or not it belongs to $\cR(\bal)$. We
will show that this happens exactly when $\Phi^{-1}(\Phi(\bal)) =
  \{\bal\}$. We shall also show that this condition holds for some
  values of~$\bal$ (in fact we already know by
  Lemma~\ref{lem:itinerary0bar} and Theorem~\ref{thm:restatement} that
  it holds for $\bal$ rational), but fails when the itinerary
  $\Phi(\bal)$ grows too rapidly.

\begin{theorem}
\label{thm:infismin}
Let $\bal\in\Delta$. Then
\begin{enumerate}[a)]
\item $\Phi^{-1}(\Phi(\bal))$ is a $d$-dimensional simplex for some
  $d$ with $0\le d\le k-2$.
\item $\IM(\bal)$ is the minimum of $\cM(\bal)$ if and only if
  $\Phi^{-1}(\Phi(\bal))$ is a point.
\end{enumerate}
\end{theorem}

\begin{proof}
Write $\bn = \Phi(\bal)$.
\begin{enumerate}[a)]
\item
The homeomorphisms $K_n^{-1}\colon\Delta\to\Delta_n$ of~(\ref{eq:KnI})
extend by the same formulae to homeomorphisms $K_n^{-1}\colon
\oD\to\oD_n\subseteq\oD$ of compact simplices. Define, for each
$r\in\N$, an embedding
\[\Upsilon_{\bn,r} = K_{n_0}^{-1}\circ K_{n_1}^{-1} \circ\cdots \circ
K_{n_r}^{-1}\colon \oD\to\oD.\] The images
$A_{\bn,r}=\Upsilon_{\bn,r}(\oD)$ of these embeddings form a
decreasing sequence of non-empty compact subsets of~$\oD$, which are
$(k-1)$-dimensional simplices since each $K_n^{-1}$ is a projectivity. 
Moreover $A_{\bn,r}\subset \Delta$ for all $r\ge k-1$, since
if~$\alpha_i>0$ for some $1\le i < k$ then
$\Upsilon_{\bn,k-1-i}(\bal)_{k-1}>0$, and therefore $\Upsilon_{\bn,
  k-i}(\bal)_k>0$: it follows that 
\[\Phi^{-1}(\bn) = \bigcap_{r\ge 0} A_{\bn,r}\]
is a non-empty compact convex subset of~$\Delta$, consisting of all
those points which have itinerary~$\bn$: this set is a simplex by
a theorem of Borovikov~\cite{Borovikov}, which states that the
intersection of a decreasing sequence of simplices is a simplex.
Since rational elements of~$\Delta$ do not share their itineraries
with any other points by Lemma~\ref{lem:itinerary0bar} and
Theorem~\ref{thm:restatement}, $\Phi^{-1}(\bn)$ cannot contain more
than one rational point, and hence has dimension at most $k-2$.

\item If $\bal\in\Q^k$ then the result follows by
  Lemma~\ref{lem:itinerary0bar} and
Theorem~\ref{thm:restatement}, so suppose that $\bal\not\in\Q^k$.
In particular $n_r>0$ for arbitrarily large~$r$, and hence
$|\Lambda_{\bn,r}(k)|\to\infty$ as $r\to\infty$.

Set
\[
\bal^{(i)}_r = \Upsilon_{\bn,r}(e^{(i)}) =
\brho(\Lambda_{\bn,r}(i)) \] for each $r\in\N$ and $1\le i\le k$, where
$e^{(i)}=(0,\ldots,0,1,0,\ldots,0)$ is the $i^\text{th}$ vertex
of~$\oD$. By compactness and convexity, $\Phi^{-1}(\Phi(\bal)) = \{\bal\}$ if and only if
$\bal^{(i)}_r \to \bal$ as $r\to\infty$ for each~$i$.

\medskip

Suppose first then that $\Phi^{-1}(\Phi(\bal)) = \{\bal\}$: we need to
show that $\Wmap(\bn)\in \cR(\bal)$.

\medskip

Let~$\epsilon>0$: we will show that $d_\infty(\brho(\Wmap(\bn)^{(m)}), \bal)
< \epsilon$ for all sufficiently large~$m$. To do this, let $R\ge 0$
be such that $d_\infty(\bal_R^{(i)}, \bal)<\epsilon/2$ for all $1\le i\le
k$, and write~$W_i$ for the word $\Lambda_{\bn,R}(i)$: thus
$d_\infty(\brho(W_i), \bal)<\epsilon/2$ for all~$i$. Let~$L=\max_{1\le i\le
  k}|W_i|$. 

Now
\[\Wmap(\bn) = \lim_{r\to\infty}\Lambda_{\bn,r}(\overline{k}) =
\lim_{r\to\infty, r> R}
\Lambda_{\bn,R}(\Lambda_{n_{R+1}}\circ\cdots\circ\Lambda_{n_r}(\overline{k}))\]
is a concatenation of the words~$W_i$. Therefore
$d_\infty(\brho(\Wmap(\bn)^{(m)}),\bal) < \epsilon$ whenever $m >
2L/\epsilon$, as required.

\medskip

Conversely, suppose that $\Wmap(\bn)\in \cR(\bal)$, so that
$\bal_r^{(k)}\to \bal$ as $r\to\infty$. We need to show that
$\Phi^{-1}(\Phi(\bal)) = \{\bal\}$, or equivalently that $\bal^{(i)}_r
\to \bal$ as $r\to\infty$ for each~$i$. The proof is by induction
on~$k\ge 2$, with the case~$k=2$ immediate since then
$\Phi^{-1}(\Phi(\bal)) = \{\bal\}$ for all~$\bal$ by~a).

We distinguish two cases. 

\begin{enumerate}[(i)]
\item Suppose first that for every~$i$ with $1\le i\le k-1$, there are
  arbitrarily large integers $r\equiv i-1\bmod k-1$ with the property
  that~$n_r>0$. 

Write $L_r^{(i)} = |\Lambda_{\bn,r}(i)|$ for each $r\in\N$ and $1\le
i\le k$, so that $ \bal^{(i)}_rL_r^{(i)}$ is an integer vector whose
entries give the number of occurences of each letter in
$\Lambda_{\bn,r}(i)$. Comparing the expressions $\Lambda_{\bn,r}(k-1)
= \Lambda_{\bn,r-1}(k\,1^{n_r+1})$ and $\Lambda_{\bn,r}(k) =
\Lambda_{\bn,r-1}(k\,1^{n_r})$ gives
\begin{equation}
\label{eq:eq1}
\bal^{(k-1)}_r = \frac{\bal_r^{(k)} L_r^{(k)} + \bal_{r-1}^{(1)} L_{r-1}^{(1)}}{L_r^{(k)}+L_{r-1}^{(1)}}.
\end{equation}
On the other hand, the first of these two expressions alone gives
\begin{equation}
\label{eq:eq2}
\bal^{(k-1)}_r = \frac{\bal_{r-1}^{(k)}L_{r-1}^{(k)} + (n_r+1)\bal_{r-1}^{(1)}L_{r-1}^{(1)}}{L_{r-1}^{(k)}+(n_r+1)L_{r-1}^{(1)}}.
\end{equation}
Solving~(\ref{eq:eq1}) and~(\ref{eq:eq2}) for  $\bal_r^{(k-1)}$  in terms of
$\bal_r^{(k)}$ and $\bal_{r-1}^{(k)}$ under the assumption $n_r>0$ gives
\[
\bal_r^{(k-1)} = 
%\frac{(n_r+1)\bal_r^{(k)}L_r^{(k)} -
%  \bal_{r-1}^{(k)}L_{r-1}^{(k)}}{(n_r+1)L_r^{(k)} - L_{r-1}^{(k)}} =
\bal_r^{(k)} + \frac{L_{r-1}^{(k)}}{(n_r+1)L_r^{(k)} - L_{r-1}^{(k)}}
\left( \bal_r^{(k)} - \bal_{r-1}^{(k)} \right).
\]
Since $\bal_r^{(k)}\to\bal$ as $r\to\infty$, it follows that for
any~$\epsilon>0$ there is some~$R$ such that
$d_\infty(\bal_r^{(k)},\bal)<\epsilon$ for all $r\ge R$, and
$d_\infty(\bal_r^{(k-1)},\bal)<\epsilon$ for all $r\ge R$ with
$n_r>0$: for $L_r^{(k)} \ge L_{r-1}^{(k)}$ for all~$r$, so that
$L_{r-1}^{(k)} \le (n_r+1)L_r^{(k)} - L_{r-1}^{(k)}$ provided that
$n_r>0$.

\medskip

Now suppose that~$r\ge R$ with $n_r>0$.  Then the expressions
$\Lambda_{\bn,r}(i) = \Lambda_{\bn,r-1}(i+1)$ for $1\le i\le k-2$ give
$\bal_{r+k-2}^{(1)} = \bal_r^{(k-1)}$ so that, by~(\ref{eq:eq1}),
\[\bal_{r+k-1}^{(k-1)} = \frac
{\bal_{r+k-1}^{(k)} L_{r+k-1}^{(k)} + \bal_r^{(k-1)} L_r^{(k-1)} }
{L_{r+k-1}^{(k)} + L_r^{(k-1)}}
\]
and hence $\bal_{r+k-1}^{(k-1)}$ is a convex combination of points
within~$\epsilon$ of~$\bal$ and so is itself within~$\epsilon$
of~$\bal$. Inductively it follows that
$d_\infty(\bal_{r+s(k-1)}^{(k-1)},\bal)<\epsilon$ for all
$s\in\N$. Since, by the defining assumption of this case, there are
$r\ge R$ with $n_r>0$ in every congruence class modulo $k-1$, we have
$d_\infty(\bal_r^{(k-1)}, \bal) < \epsilon$ for all sufficiently
large~$r$. Therefore $\bal_r^{(k-1)}\to\bal$ as $r\to\infty$.

\medskip

 Since $\bal_r^{(i)} = \bal_{r+i-(k-1)}^{(k-1)}$ for all $1\le i\le
 k-2$ and $r\ge (k-1)-i$, it follows that \mbox{$\bal_r^{(i)}\to\bal$} as
 $r\to\infty$ for all~$i$ as required.

\medskip\medskip

\item Suppose then that there is some $i$ with $1\le i\le k-1$ such
  that~$n_r=0$ for all sufficiently large $r\equiv i-1 \bmod k-1$. We
  shall show that if $\Phi(\bbeta) = \bn$ then $\bbeta = \bal$. Now
  for each~$r\in\N$ we have that $\Phi(\bbeta) = \bn$ if and only if
  $\bbeta = K_{n_0}^{-1}\circ K_{n_1}^{-1}\circ\cdots\circ
  K_{n_{r-1}}^{-1}(\bbeta')$ for some~$\bbeta'$ with
  $\Phi(\bbeta')=n_rn_{r+1}\ldots$, so we can suppose without loss of
  generality that $n_r=0$ for every \mbox{$r\equiv 0 \bmod k-1$}, and
  hence by Lemma~\ref{lem:zero-component} that
  $\alpha_1=\beta_1=0$. By Remark~\ref{rmk:reduce-k} (and using the
  notation introduced there), $\bm:=\Phi(\pi_1(\bal)) =
  \Phi(\pi_1(\bbeta))$ is obtained from~$\bn$ by deleting the zero
  entries in positions which are multiples of~$k-1$. 

Write $\bal_r = \bal_r^{(k)}$ and $\bal_r' =
\Upsilon_{\bm,r}(0,0,\ldots,0,1) \in \Delta^{k-1}$ for
each~$r\in\N$. We shall show that
\begin{equation}
\label{eq:goal}
\bal'_{s(k-2)+i} = \pi_1(\bal_{s(k-1)+i+1})\quad\text{ whenever $0\le i\le
  k-3$ and $s\in\N$}.
\end{equation}
That is, the proportions of letters in each $\Lambda_{\bm,r}(k-1)$ (a
sequence over $k-1$ letters) is obtained from the proportions of
letters in $\Lambda_{\bn,r'}(k)$ by deleting an initial zero,
where~$r'$ is an appropriate index which increases with~$r$.  This
will establish the result. For then $\bal_r\to\bal$ implies
$\bal'_r\to\pi_1(\bal)$, or in other words
$S(\bm)\in\cR(\pi_1(\bal))$. Hence $\pi_1(\bal)=\pi_1(\bbeta)$ by the
inductive hypothesis, so that $\bal=\bbeta$ as required.

Observe first that when~$s=0$, equation~(\ref{eq:goal}) reads
\begin{equation}
\label{eq:simple-case}
K_{m_0}^{-1}\circ K_{m_1}^{-1}\circ\cdots\circ K_{m_i}^{-1}
(0,\ldots,0,1) = \pi_1(K_0^{-1}\circ K_{m_0}^{-1}\circ
K_{m_1}^{-1}\circ\cdots\circ K_{m_i}^{-1} (0,0,\ldots,0,1))
\end{equation}
for $0\le i\le k-3$, where on the left hand side
$(0,\ldots,0,1)\in\Delta^{k-1}$, and on the right hand side
$(0,0,\ldots,0,1)\in\Delta^k$. This is a straightforward consequence
of~(\ref{eq:KnI}): since \mbox{$i\le k-3$}, the $(k-1)^\text{th}$
component of $K_{m_0}^{-1}\circ
K_{m_1}^{-1}\circ\cdots\circ K_{m_i}^{-1} (0,0,\ldots,0,1)
\in\Delta^k$ is zero, so that applying $K_0^{-1}$ cyclically permutes
the first $k-1$ components.

Now it follows from~(\ref{eq:KnI}) that
 \[
      \pi_{i+1}^{-1}\circ K_m^{-1} = K_m^{-1}\circ\pi_i^{-1} \,\colon
     \Delta^{k-1} \to \Delta^{k,i+1}
  \]
for all $1\le i\le k-2$ and all $m\in\N$. Applying this for
$i=1,2,\ldots,k-2$ in succession gives
  \[ \pi_{k-1}^{-1}\circ
  K_{m_{k-2}}^{-1}\circ \cdots \circ K_{m_{1}}^{-1} =
  K_{m_{k-2}}^{-1}\circ \cdots \circ
  K_{m_{1}}^{-1}\circ\pi_1^{-1}\,\colon\Delta^{k-1}\to\Delta^{k,k-1}\]
  for all $m_1,\ldots,m_{k-2}$. Then, using again the observation that
  if $\bal\in\Delta^{k,k-1}$ then $K_0^{-1}$ cyclically permutes its
  components,
\begin{equation}
\label{eq:gen-case}
 K_{m_{k-2}}^{-1}\circ \cdots \circ
K_{m_1}^{-1} 
 =
\pi_1 \circ K_0^{-1}\circ K_{m_{k-2}}^{-1}\circ\cdots \circ
K_{m_1}^{-1}\circ\pi_1^{-1}
\,\colon \Delta^{k-1}\to\Delta^{k-1}
\end{equation}
for all $m_1,\ldots,m_{k-2}$. 

Applying~(\ref{eq:simple-case}) followed by $s$ applications
of~(\ref{eq:gen-case}) establishes~(\ref{eq:goal}) as required.
\end{enumerate}

\end{enumerate}

\end{proof}

\begin{remarks} \mbox{}
\label{rmk:infismin}
\begin{enumerate}[a)]
\item It is clear that if $\Phi^{-1}(\bn)$ is more than just one point, then
$\Wmap(\bn)$ can only be the minimum of $\cM(\bal)$ for at most one
$\bal\in\Phi^{-1}(\bn)$. The content of the final part of the proof is
that in fact it is not the minimum of any of the sets $\cM(\bal)$, and
indeed does not belong to $\cR(\bal)$ for any $\bal\in\Delta$.
\item Combining Theorem~\ref{thm:infismin}~b) with Lemma~17
  of~\cite{BT} and Lemma~4.2 of~\cite{Bruin} yields the following
  result: the action of the shift map on the orbit closure
  $\Sigma_{\bal}$ of Remarks~\ref{rmk:primed-and-almost-periodic}~b)
  is uniquely ergodic if and only if $\bal$ is regular.
\end{enumerate}
\end{remarks}

In view of this result, we make the following definitions:
\begin{defns}
$\bal\in\Delta$ is {\em regular} if
$\Phi^{-1}(\Phi(\bal)) = \{\bal\}$, and {\em exceptional} otherwise.
\end{defns}

When~$k=2$, every $\bal\in\Delta$ is regular by
Theorem~\ref{thm:infismin}~a). Therefore, in the
two letter case, there is an \mbox{$\bal$-minimax} sequence for
all~$\bal$: these are the well-known Sturmian
sequences~\cite{Sturmian1,Sturmian2}. When $k\ge 3$, we have already seen that $\bal$ is regular
if it is rational (i.e. if $\Phi(\bal)_r=0$ for all sufficiently
large~$r$). The following theorem states that the same is true when
$\Phi(\bal)_r>0$ grows at most quadratically with~$r$, and, on the other
hand, that if $\Phi(\bal)_r$ grows too fast then $\bal$ is exceptional.

\begin{theorem}
\label{thm:regular}
Let~$\bal\in\Delta$ and $\bn = \Phi(\bal)$.
\begin{enumerate}[a)]
\item If there is some~$C$ such that $0< n_r \le Cr^2$ for all~$r$, then
  $\bal$ is regular.
\item If $k\ge 3$ and $n_r \ge 2^{r+2}\prod_{i=0}^{r-1}(n_i+2)$ for
  all~$r\ge 1$, then $\Phi^{-1}(\Phi(\bal))$ is a simplex of
  dimension~$k-2$, so that $\bal$ is exceptional.
\end{enumerate}
\end{theorem}

\begin{proof}
We use the notation of the proof of Theorem~\ref{thm:infismin}.
\begin{enumerate}[a)]
\item 
We will use a theorem of Birkhoff~\cite{Birkhoff,Carroll} to show that
$\Delta$ is contracted by the embeddings~$\Upsilon_{\bn,r}$, and we
start by giving some necessary definitions and stating this
theorem. Let~$A = (a_{ij})$ be a $k$ by $k$ matrix with strictly
positive entries, and $f_A$ be its projective action on~$\oD$:
that is, $f_A\colon\oD\to\oD$ is defined by
\[f_A(\bal) = \frac{A\,\bal}{||A\,\bal||_1}.\]
Define also
\begin{equation}
\label{eq:dA}
d(A) = \max_{1\le i,j,l,m\le k}\,\, \frac{a_{il}a_{jm}}{a_{im}a_{jl}}
\ge 1
\end{equation}
 (that is, $d(A)$ is the largest number that can be obtained by
choosing four elements of~$A$ arranged in a rectangle, and dividing the
product of the two elements on one diagonal by the product of the two
elements on the other). $d(A)$ is stricly greater than one unless~$A$
has rank~$1$.

Let $\tau\colon[1,\infty)\to[0,1)$ be the strictly increasing function
    $\tau(d) = (\sqrt{d}-1)/(\sqrt{d}+1)$. Write~$\mathring{\Delta}$
    for the simplex~$\Delta$ less its faces, and
    let~$\delta\colon\mathring{\Delta}\times\mathring{\Delta}\to \R_{\ge0}$ be
    Hilbert's projective metric (which generates the
    Euclidean topology),
\[
\delta(\bal,\bbeta) = \log \max_{1\le i,j\le k}\,\,\frac{\alpha_i\,\beta_j}{\alpha_j\,\beta_i}.
\]

Birkhoff's theorem states that, provided $d(A)>1$, the restriction
of~$f_A$ to $\mathring{\Delta}$ contracts the metric~$\delta$ by $\tau(d(A))$:
that is, $\delta(f_A(\bal), f_A(\bbeta)) \le \tau(d(A))\, \delta(\bal,
\bbeta)$ for all $\bal, \bbeta \in \mathring{\Delta}$.

Now let~$A(n)$ be the $k$ by $k$ matrix with $A(n)_{1,k-1}=n+1$,
$A(n)_{1,k}=n$, $A(n)_{i, i-1} = 1$ for $2\le i\le k$, $A(n)_{k,k}=1$,
and all other entries zero: as an example, when~$k=5$, 
\[
A(n) = \left(\begin{array}{ccccc}
 0 & 0 & 0 & n+1 & n\\
 1 & 0 & 0 & 0 & 0\\
 0 & 1 & 0 & 0 & 0\\
 0 & 0 & 1 & 0 & 0\\
 0 & 0 & 0 & 1 & 1
\end{array}\right).
\]
By~(\ref{eq:KnI}), we have $f_{A(n)} =
K_n^{-1}\colon\oD\to\oD$. Although~$A(n)$ has some zero entries, we
shall see that any product of $2k-3$ such matrices $A(n_r)$ with each
$n_r>0$ is strictly positive.

Write $A(n_0,\ldots,n_r) = \prod_{s=0}^r A(n_s)$. By considering the
action of $\Lambda_{n_0}\circ\cdots\circ\Lambda_{n_{k-3}}$ on each of
the letters $1,\ldots,k$, it can be seen that $A(n_0,\ldots,n_{k-3})$
has row~$i$, for $1\le i\le k-2$, consisting of $i$ zeros followed by
$n_{i-1}+1$ and then $n_{i-1}$ in the other columns; row $k-1$ has a $1$ in column
$1$ and zeros in the other columns; and row $k$ has a zero in column
$1$ and $1$s in the other columns. Similarly $A(n_{k-2},\ldots,
n_{2k-4}$) has row~$i$, for $1\le i\le k-1$, consisting of $i-1$ zeros
followed by $n_{k-3+i}+1$ on the diagonal and $n_{k-3+i}$ in the other
columns; while row~$k$ has $1$ in every column. As an example, when
$k=5$, these two matrices are given by
\[
\left(\begin{array}{ccccc}
0 & n_0 + 1 & n_0 & n_0 & n_0 \\
0 & 0 & n_1+1 & n_1 & n_1\\
0 & 0 & 0 & n_2+1 & n_2\\
1 & 0 & 0 & 0 & 0\\
0 & 1 & 1 & 1 & 1
\end{array}\right) \quad\text{and}\qquad
\left(\begin{array}{ccccc}
n_3+1 & n_3 & n_3 & n_3 & n_3 \\
0 & n_4+1 & n_4 & n_4 & n_4 \\
0 & 0 & n_5+1 & n_5 & n_5\\
0 & 0 & 0 & n_6+1 & n_6\\
1 & 1 & 1 & 1 & 1
\end{array}\right).
\]

The product $A(n_0,\ldots,n_{2k-4}) = (a_{ij})_{1\le i,j\le k}$ of
these matrices is therefore strictly positive when each~$n_r>0$, with
each~$a_{ij}$ a polynomial of degree at most~2 in
$n_0,\ldots,n_{2k-4}$. We shall show that, for each~$1\le i\le k$ and
each $1\le l < m \le k$, the quotient $a_{il}/a_{im}$ is bounded above
by 2, while the quotient $a_{im}/a_{il}$ is bounded above by a linear
function of $n_{k-1}, \ldots, n_{2k-4}$. As a consequence,
since~(\ref{eq:dA}) says that $d(A)$ is the product of one quotient of
the first type and one of the second, there is some~$R$, depending
only on~$k$, such that
\begin{equation}
\label{eq:bound}
d(A(n_0,\ldots,n_{2k-4})) \le R(n_{k-1}+\cdots + n_{2k-4})
\end{equation}
provided that each~$n_r>0$.

The claim is straightforward when~$i=k-1$, in which case $a_{il}$ is
either $n_{k-2}$ or $n_{k-2}+1$; and when~$i=k$, in which case
$a_{i1}=1$, $a_{il} = 2+\sum_{j=k-1}^{k-3+l}n_j$ for $2\le l\le k-1$,
and $a_{ik} = a_{i,k-1}-1$. When $1\le i\le k-2$, the explicit
descriptions of the elements of $A(n_0,\ldots,n_{k-3})$ and
$A(n_{k-2},\ldots,n_{2k-4})$ give
\[
a_{il} = 
\begin{cases}
n_{i-1} & \text{if }1\le l\le i,\\
n_{i-1}(n_{k+i-2}+2) + (n_{k+i-2}+1) & \text{if }l = i+1,\\
n_{i-1}\left(2+\sum_{j=k+i-2}^{k+l-3}n_j\right) + n_{k+i-2} & \text{if
} i+2 \le l < k,\\
n_{i-1}\left(1+\sum_{j=k+i-2}^{2k-4}n_j\right) + n_{k+i-2} & \text{if
} l=k,\\
\end{cases}
\]
from which the claim follows.

Now let~$\bal\in\Delta$, and suppose that there is some~$C$ such
that~$\bn=\Phi(\bal)$ satisfies $0<n_r\le Cr^2$ for all~$r$. For
each~$r\ge 0$ we have
\begin{multline*}
\Upsilon_{\bn, (r+1)(2k-3)-1}(\oD) = (K_{n_0}^{-1}\circ \cdots \circ
K_{n_{2k-4}}^{-1})\circ(K_{n_{2k-3}}^{-1}\circ \cdots \circ
K_{n_{4k-7}}^{-1}) \circ\cdots \circ \\ (K_{n_{r(2k-3)}}^{-1} \circ \cdots
\circ K_{n_{(r+1)(2k-3)-1}}^{-1})(\oD)
\end{multline*}
 Since $(K_{n_{r(2k-3)}}^{-1} \circ \cdots \circ
 K_{n_{(r+1)(2k-3)-1}}^{-1})(\oD) \subset\mathring{\Delta}$ (because
 the product of $2k-3$ matrices $A(n)$ is strictly positive), it is
 enough to show that
\[\prod_{r=0}^\infty \tau(d(A(n_{r(2k-3)}, \ldots, n_{(r+1)(2k-3)-1}))) =
0.\]

By~(\ref{eq:bound}) and $n_r\le Cr^2$, there is some~$Q$ depending
only on~$C$ and~$k$ such that $d_r := d(A(n_{r(2k-3)}, \ldots,
n_{(r+1)(2k-3)-1})) \le (Qr)^2$ for all $r\ge1$, so that $\tau(d_r)
\le (Qr-1)/(Qr+1)$. Recall that if $0<a_r\le 1$ for all~$r$
then~$\prod_{r=0}^\infty a_r = 0$ if and only if $\sum_{r=0}^\infty
\left(\frac{1}{a_r}-1\right)$ diverges. Since
\[\frac{1}{\tau(d_r)} - 1 \ge \frac{2}{Qr-1},\]
the result follows.

\medskip

\item Set
\[\delta_r = \min_{1\le i<j\le k-1} d_\infty(\bal^{(i)}_r,
\bal^{(j)}_r)\] for each~$r\ge 0$, the smallest distance between a
pair of vertices in the simplex $A_{\bn,r}$ excluding the vertex
$\Upsilon_{\bn,r}(0,0,\ldots,0,1)$. We shall show that $\delta_0=1$
and $\delta_r \ge \delta_{r-1}-1/2^{r+2}$ for each $r\ge 1$, so that
$\delta_r>3/4$ for all~$r$. It is therefore not possible for all
of the $\bal^{(j)}_r$ to converge to the same point.

That $\delta_0=1$ is straightforward, since $\bal^{(i)}_0 =
K_{n_0}^{-1}(e^{(i)})$ is equal to $e^{(i+1)}$ if $1\le i\le k-2$, and
to $((n_0+1)e^{(1)} + e^{(k)})/(n_0+2)$ if $i = k-1$.

Now let~$r\ge 1$. If $1\le i\le k-2$ then we have
$\Lambda_{\bn,r}(i) = \Lambda_{\bn,r-1}(\Lambda_{n_r}(i)) =
\Lambda_{\bn,r-1}(i+1)$, so that
\[\bal_r^{(i)} = \bal_{r-1}^{(i+1)} \qquad\text{for }1\le i\le k-2.\]

Consider then the case $i=k-1$. By~(\ref{eq:eq2})
\[\bal_r^{(k-1)} = \frac
{\bal_{r-1}^{(k)}L_{r-1}^{(k)} \,+\, (n_r+1)\bal_{r-1}^{(1)}L_{r-1}^{(1)}}
{L_{r-1}^{(k)} \,+\, (n_r+1)L_{r-1}^{(1)}},
\]
so that
\[\bal_r^{(k-1)} - \bal_{r-1}^{(1)} = 
\frac
{L_{r-1}^{(k)}
\left(\bal_{r-1}^{(k)} - \bal_{r-1}^{(1)} \right)}
{(n_r+1)L_{r-1}^{(1)} + L_{r-1}^{(k)}},
\]
in which each component has absolute value bounded above
by $\prod_{i=0}^{r-1}(n_i+2)/n_r \le 1/2^{r+2}$, using
$L_{r-1}^{(k)}\le\prod_{i=0}^{r-1}(n_i+2)$ in the numerator and
$L_{r-1}^{(i)}\ge 1$ in the denominator.

Therefore $d_\infty(\bal_r^{(k-1)}, \bal_{r-1}^{(1)})\le 1/2^{r+2}$, and
we saw in the first part of the proof that $d_\infty(\bal_r^{(i)},
\bal_{r-1}^{(i+1)})=0$ for $1\le i\le k-2$. This gives $\delta_r
\ge \delta_{r-1} - 1/2^{r+2}$ as required.

To show that $\Phi^{-1}(\Phi(\bal))$ is a simplex of
dimension~$k-2$, let $\pi\colon\R^k\to\R^{k-1}$ be projection onto the
first~$k-1$ coordinates. Then 
\[V_0:= \{\pi(\bal_0^{(i)})\,:\,1\le i\le k-1\} = \{(n_0+1)\pi(e^{(1)})/(n_0+2),
\pi(e^{(2)}), \pi(e^{(3)}), \ldots, \pi(e^{(k-1)})\},\]
and $(n_0+1)/(n_0+2) \ge 1/2$.
Now for each~$r\ge 1$, the set $V_r := \{\pi(\bal_r^{(i)})\,:\,1\le
i\le k-1\}$ is within $d_\infty$-Hausdorff distance~$1/4$ of~$V_0$,
and hence the same is true for the limit $V_\infty$. The $k-1$ points
of $V_\infty$ therefore span a simplex of dimension~$k-2$, which is
the $\pi$-image of a simplex of dimension~$k-2$ contained in
$\Phi^{-1}(\Phi(\bal))$. 

\end{enumerate}

\end{proof}

\begin{example}
The conditions of Theorem~\ref{thm:regular}a) are obviously satisfied
when $\bn = \Phi(\bal) = \overline{n_0\ldots n_{r-1}}$ is periodic
without any zero entries: by the theorem, such a sequence is the
itinerary of a unique periodic point of~$K$. The corresponding minimax
sequence $\IM(\bal)$ is the fixed point of the substitution
$\Lambda_{n_0}\circ\cdots\circ\Lambda_{n_{r-1}}$, and therefore
generates a substitution minimal set~\cite{SMS}.

The simplest such example is when~$k=3$ and~$\Phi(\bal) = 
\overline{1}$. The minimum of $\cM(\bal)$ is then given by
\[\lim_{r\to\infty}\Lambda_1^r(3) =
3123113122312311311312311312231223123113122312311311312311311312
\ldots\,,\] 
the unique fixed point of~$\Lambda_1$. In this example
$\bal$ is the unique fixed point of~$K_1$ or, equivalently, the
(suitably normalized)  strictly positive eigenvector of the matrix
\[A(1) = \left(\begin{array}{ccc}
  0 & 2 & 1\\
 1 & 0 & 0 \\
 0 & 1 & 1 
\end{array}\right)
\]
from the proof of Theorem~\ref{thm:regular}a). Notice that the minimum
of $\cM(\bal)$ is not of Arnoux-Rauzy type~\cite{GBA}: for example, it
has six factors of length two, and the substitution~$\Lambda_1$ is not
Pisot. 
\end{example}

\begin{remarks} \mbox{}
\label{rmk:regular}
\begin{enumerate}[a)]
\item Theorems~11 and~12 of~\cite{BT} improve substantially on
  Theorem~\ref{thm:regular} in the case~$k=3$ (only). After
  translation to the notation used here, they read:

{\sc Theorem (Bruin and Troubetzkoy)}\,\,
\em
Let~$k=3$, and let $\bal\in\Delta$ and $\bn=\Phi(\bal)$.
\begin{itemize}
\item For each~$r\ge 0$, let $L_{2r}=\min\{s\ge
  1\,:\,n_{2r+s}\not=0\}$. If either
\begin{eqnarray*}
\sum_r \frac{n_{2r}}{n_{2r}+1} \sqrt{\frac{1}{(n_{2r-1}+1)L_{2r}}} &=&
\infty, \quad\text{ or}\\
\prod_r \frac{n_{2r}+1}{n_{2r-1} + 1 + \frac{1}{L_{2r}}} &=& 0,
\end{eqnarray*}
or if either condition holds for the shift $\sigma(\bn) = \Phi(K(\bal))$ of $\bn$,
then $\bal$ is regular.
\item If there is some~$\lambda>1$ such that $n_{r+1}\ge \lambda n_r$
  for all sufficiently large~$r$, then $\bal$ is exceptional.
\end{itemize}
\rm This result gives rise to a striking pair of examples: on the one
hand, if $\Phi(\bal)_r = 2^r$ for all~$r$ then $\bal$ is exceptional
by the second statement; while on the other hand, if
$\Phi(\bal)_r=2^r$ when $r$ is even and $\Phi(\bal)_r=3^r$ when $r$ is
odd, then $\bal$ is regular by the second condition in the former
statement.

\item The result of Theorem~\ref{thm:regular}a) clearly extends to the
  case where finitely many of the~$n_r$ are zero. When $n_r=0$ for
  arbitrarily large~$r$ the situation is more complicated, as the
  product of $2k-3$ successive matrices need not be strictly
  positive. This can not always be remedied by grouping the sequence
  of matrices more judiciously: in the case where $n_{r(k-1)} = 0$ for
  all~$r$, no product $A(n_s, n_{s+1}, \ldots, n_{s+t})$ is strictly
  positive. This case arises when considering the itinerary of an
  element~$\bal$ of~$\Delta$ which has some zero coordinates
  (Lemma~\ref{lem:zero-component}), and can be treated by induction
  on~$k$.
\item The fact that the bound of~(\ref{eq:bound}) depends only on
  $k-2$ of the $2k-3$ variables means that it is sufficient for
  regularity to have control over the~$n_r$ along an appropriate
  subsequence. 
\item The growth condition in Theorem~\ref{thm:regular}b) --- which,
  for example, is satisfied by $n_r = 2^{2^{3r}}$ --- could easily be
  improved by improving the bounds on $L_{r-1}^{(k)}$ and
  $L_{r-1}^{(i)}$ in the penultimate paragraph of the proof:
  the point here is simply to show that exceptional~$\bal$ exist. In
  fact, numerical experiments suggest that, when~$k=3$,
  $\Phi^{-1}(\bn)$ is a non-trivial interval when $n_r=r^3$, so that
  even the $k=3$ results of Bruin and Troubetzkoy are far from
  optimal. 
\end{enumerate}
\end{remarks}

We finish by showing -- closely following the proof of Corollary~13
of~\cite{BT} -- that a generic element $\bn$ of $\N^\N$ is the
itinerary of only one point. We use the following lemma.

\begin{lemma}
\label{lem:non-expanding}
For all $n\ge 0$, the map
$K_n^{-1}\colon\mathring{\Delta}\to\mathring{\Delta}$ does not expand
the Hilbert metric: that is, $\delta(K_n^{-1}(\bal), K_n^{-1}(\bbeta))
\le \delta (\bal, \bbeta)$ for all $\bal, \bbeta\in\mathring{\Delta}$.
\end{lemma}
\begin{proof}
Let $\bal,\bbeta\in\mathring{\Delta}$, and write
\[\bal' := K_n^{-1}(\bal) = C\,\,((n+1)\alpha_{k-1}+ n\alpha_k,\,
\alpha_1,\, \alpha_2,\,
\ldots,\,\alpha_{k-2},\,\alpha_{k-1}+\alpha_k),\] where $C = C(\bal)$
is a constant, and similarly $\bbeta' := K_n^{-1}(\bbeta)$. To prove
the lemma, we need to show that whenever $1\le i<j\le k$, there exist
$I$ and $J$ between $1$ and $k$ with
\[\frac{\alpha_i'\,\beta_j'}{\beta_i'\,\alpha_j'} \,\le\,
\frac{\alpha_I\,\beta_J}{\beta_I\,\alpha_J}.\] 
This can be established straightforwardly by cases, using the
elementary fact that if $a,b,c,d$ are positive reals then
$(a+b)/(c+d)$ lies between $a/c$ and $b/d$.
\begin{itemize}
\item If $i$ and $j$ are both between $2$ and $k-1$ then
  $\displaystyle{\frac{\alpha_i'\beta_j'}{\beta_i'\alpha_j'} =
  \frac{\alpha_{i-1}\beta_{j-1}}{\beta_{i-1}\alpha_{j-1}}}$. 
\item If $i=1$ and $j<k$ then $\displaystyle{
  \frac{((n+1)\alpha_{k-1} + n\alpha_k) \beta_{j-1}}{
    ((n+1)\beta_{k-1} + n\beta_k)\alpha_{j-1}}}$ lies between
  $\displaystyle{\frac{(n+1)\alpha_{k-1}\beta_{j-1}}{(n+1)\beta_{k-1}\alpha_{j-1}}
  = \frac{\alpha_{k-1}\beta_{j-1}}{\beta_{k-1}\alpha_{j-1}}}$ and
  $\displaystyle{\frac{n\alpha_k\beta_{j-1}}{n\beta_k\alpha_{j-1}} =
  \frac{\alpha_k\beta_{j-1}}{\beta_k\alpha_{j-1}}}$.
\item If $i>1$ and $j=k$ then the argument is identical, except that
  the factors $n+1$ and $n$ are omitted.
\item If $i=1$ and $j=k$ then $\displaystyle{\frac{((n+1)\alpha_{k-1}
    + n\alpha_k) (\beta_{k-1}+\beta_k)}{ ((n+1)\beta_{k-1} +
    n\beta_k)(\alpha_{k-1}+\alpha_k)}}$ lies between
  $\displaystyle{\frac{((n+1)\alpha_{k-1} + n\alpha_k)\beta_{k-1}}{
    ((n+1)\beta_{k-1} + n\beta_k)\alpha_{k-1}}}$ and
  $\displaystyle{\frac{((n+1)\alpha_{k-1} + n\alpha_k) \beta_k}{
    ((n+1)\beta_{k-1} + n\beta_k)\alpha_k}}$, each of which is between
  two terms of the required type by the argument above.
\end{itemize}
\end{proof}

Let $\cO\subset \N^\N$ be the set of itineraries
$\bn$ which contain infinitely many disjoint subwords $1^{2k-3}$, and
let $\Reg\subset \N^\N$ be the set of regular itineraries $\bn$,
i.e. those for which $\Phi^{-1}(\bn)$ is a point.

\newpage %FORMAT

\begin{theorem}
\label{thm:generic-regular}
$\cO\subset\Reg$, and $\cO$ is a dense $G_\delta$ subset of $\N^\N$.
\end{theorem}
\begin{proof}
Let $\bn\in\cO$.  As shown in the proof of Theorem~\ref{thm:regular},
the map $K_1^{-(2k-3)}\colon \mathring{\Delta}\to\mathring{\Delta}$ is
represented by the strictly positive matrix $A(1)^{2k-3}$, and hence
by Birkhoff's theorem contracts the Hilbert metric by a factor
$\lambda\in(0,1)$. It follows, using Lemma~\ref{lem:non-expanding},
that if $r_i$ is the index of the start of the $i^\text{th}$ disjoint
subword $1^{2k-3}$ in $\bn$, then we have
\[A_{\bn, r_i+(2k-3)} = \Upsilon_{\bn, r_i+(2k-3)}(\oD) = \Upsilon_{\bn,
  r_i}\left(K_1^{-(2k-3)}(\oD)\right)\] has diameter bounded above by
$\lambda^{i-1}D$, where $D$ is the Hilbert diameter of
$K_1^{-(2k-3)}(\oD)\subset \mathring{\Delta}$. Therefore
$\Phi^{-1}(\bn)$ is a single point, so that $\bn\in\Reg$.

For each $N\ge 0$ let $\cO_N\subset\N^\N$ be the set of itineraries
which contain a word $1^{2k-3}$ starting after the $N^\text{th}$
symbol. Then $\cO_N$ is open and dense in the Baire space $\N^\N$, so
that $\cO=\bigcap_{N\ge 0} \cO_N$ is a dense~$G_\delta$ subset of
$\N^\N$ as required. 
\end{proof}

\begin{acknowledgements} 
We are grateful to Arnaldo Nogueira for helpful conversations.
\end{acknowledgements}

\bibliography{lexrefs}

\end{document}